\documentclass[11pt]{amsart}

\usepackage{tikz, float}
\usetikzlibrary{patterns}
\usepackage{amssymb}
\usepackage{url}
\usepackage{amscd}
\usepackage{amsthm}
\usepackage{amsmath}
\usepackage{graphicx}

\newtheorem {Lemma} {Lemma}
\newtheorem {Theorem}  {Theorem}

\newtheorem {Problem} {Problem}

\begin{document}
\baselineskip = 15pt
\bibliographystyle{plain}

\title[Undecidability of Translational Tiling with Orthogonally Convex Polyominoes]{Undecidability of Translational Tiling of the Plane with Orthogonally Convex Polyominoes}
\date{}
\author{Chao Yang}\address{School of Mathematics and Statistics, Guangdong University of Foreign Studies, Guangzhou, 510006, China} 
\email{sokoban2007@163.com, yangchao@gdufs.edu.cn}

\author{Zhujun Zhang}
\address{Big Data Center of Fengxian District, Shanghai, 201499, China}
\email{zhangzhujun1988@163.com}

\begin{abstract}
The first undecidability result on the tiling is the undecidability of translational tiling of the plane with Wang tiles, where there is an additional color matching requirement. Later, researchers obtained several undecidability results on translational tiling problems where the tilings are subject to the geometric shapes of the tiles only. However, all these results are proved by constructing tiles with extremely concave shapes. It is natural to ask: can we obtain undecidability results of translational tiling with convex tiles? Towards answering this question, we prove the undecidability of translational tiling of the plane with a set of $7$ orthogonally convex polyominoes.
\end{abstract}

\maketitle


\section{Introduction} \label{sec_intro}

The study of undecidability of translational tiling problems originated from the domino problem introduced by Hao Wang in the 1960s. A \textit{Wang tile} is a unit square with each edge assigned a color. Given a finite set of Wang tiles (see Figure \ref{fig_wang_set} for an example), Wang considered the problem of tiling the entire plane with translated copies of the set, under the conditions that the tiles must be edge-to-edge and the color of common edges of any two adjacent Wang tiles must be the same \cite{wang61}. This is known as \textit{Wang's domino problem}. Berger showed that Wang's domino problem is undecidable in general (i.e. the size of the set of Wang tiles can be arbitrarily large\footnote{If the size of the set of Wang tiles is fixed, then Wang's domino problem is decidable, as there are only a finite number of instances.}) in the 1960s.

\begin{Theorem}[\cite{b66}]\label{thm_berger}
    Wang's domino problem is undecidable.
\end{Theorem}


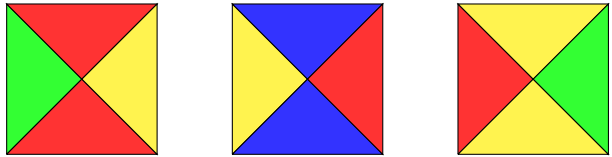
\begin{figure}[H]
\begin{center}
\begin{tikzpicture}

\draw [fill=green!80] (0,0)--(1,1)--(0,2)--(0,0);
\draw [fill=red!80] (0,0)--(2,0)--(0,2)--(2,2)--(0,0);
\draw [fill=yellow!80] (1,1)--(2,2)--(2,0)--(1,1);

\draw [fill=yellow!80] (3+0,0)--(3+1,1)--(3+0,2)--(3+0,0);
\draw [fill=blue!80] (3+0,0)--(3+2,0)--(3+0,2)--(3+2,2)--(3+0,0);
\draw [fill=red!80] (3+1,1)--(3+2,2)--(3+2,0)--(3+1,1);

\draw [fill=red!80] (6+0,0)--(6+1,1)--(6+0,2)--(6+0,0);
\draw [fill=yellow!80] (6+0,0)--(6+2,0)--(6+0,2)--(6+2,2)--(6+0,0);
\draw [fill=green!80] (6+1,1)--(6+2,2)--(6+2,0)--(6+1,1);

\end{tikzpicture}
\end{center}
\caption{A set of $3$ Wang tiles.}\label{fig_wang_set}
\end{figure}

Note that in Wang's domino problem, the matching rule is mainly non-geometrical (i.e., by colors). In this paper, we study the undecidability of the following translational tiling problems (Problem \ref{pro_main}) of the $n$-dimensional space $\mathbb{Z}^n$, where the matching rule is purely geometric. An interesting trend of the recent study on Problem \ref{pro_main} focuses on minimizing the number of tiles, but maintaining the undecidability. 

\begin{Problem}[Translational tiling of $\mathbb{Z}^n$ with a set of $k$ tiles] \label{pro_main}
A tile is a finite subset of $\mathbb{Z}^n$. Let $k$ and $n$ be fixed positive integers. Given a set $S$ of $k$ tiles in $\mathbb{Z}^n$, is there an algorithm to decide whether $\mathbb{Z}^n$ can be tiled by translated copies of tiles in $S$?
\end{Problem}

It is straightforward that if the above translational tiling problem is undecidable for a fixed pair of parameters $n_0$ and $k_0$, then it is also undecidable for $(n,k)$ for all $n\geq n_0$ and $k\geq k_0$. The answer to Problem~\ref{pro_main} is known for some pairs of parameters $n$ and $k$. The problem is decidable for $n=1$ \cite{s93}, and for $(n,k)=(2,1)$ \cite{bn91,b20, gt21, w15}. On the other hand, the problem is undecidable for $n=2$ and $k\geq 7$ \cite{yang23, yang23b, yz24,yz24e}, for $n=3$ and $k\geq 4$, and for $n=4$ and $k\geq 3$ \cite{yz24b,yz24c,yz24d}. Two recent results of Greenfeld and Tao \cite{gt24a, gt24b} suggest that the translational tiling problem may be undecidable even for one tile ($k=1$) and some sufficiently large fixed dimension $n$. In fact, they disprove the periodic tiling conjecture \cite{gs16, lw96, s74} by showing the existence of an aperiodic monotile in some extremely large fixed dimension \cite{gt24a}. They also show that if the dimension $n$ is part of the input, the translational tiling for subsets of $\mathbb{Z}^n$ with one tile is undecidable \cite{gt24b}. 

However, all the results mentioned in the previous paragraph depend on constructing tiles with extremely concave shapes, where the concave shapes are crucial in the proof of undecidability. Can the translational tiling problem be undecidable with convex tiles for some pair of parameters $(n,k)$? 

A polytope is said to be \textit{convex} if the intersection with an arbitrary line is either a segment (including the degenerate segment, i.e., a point) or empty. Because a convex polyomino must be a rectangle, therefore any convex polyomino can tile the plane, and the problem is trivially decidable. So we consider a weaker notion of convexity in this paper. There are several different notations of restricted convexity in the literature; see \cite{fw04} for a comprehensive survey. We only consider the concept of orthogonally convex in this paper. A polytope is said to be \textit{orthogonally convex} if, for every line $L$ that is parallel to one of the standard basis vectors, the intersection of the polytope with $L$ is either a segment or empty. Figure \ref{fig_pento} illustrates two pentominoes: the left one is orthogonally convex, and the right one is not orthogonally convex. Orthogonal convexity is also known as \textit{rectilinear convexity} in dimension $2$. Orthogonal convexity has received considerable studies in the literature. For example, a rectilinear convex variant of the Erd\H{o}s-Szekeres problem has been studied in \cite{g19}.


\begin{figure}[H]
\begin{center}
\begin{tikzpicture}

\draw (2,0)--(2,3)--(1,3)--(1,0)--(2,0);
\draw (0,1)--(0,2)--(3,2)--(3,1)--(0,1);

\draw (8,0)--(8,3)--(9,3)--(9,0)--(8,0);
\draw (8,0)--(7,0)--(7,1)--(9,1); 
\draw (8,3)--(7,3)--(7,2)--(9,2); 

\end{tikzpicture}
\end{center}
\caption{Orthogonally convex and not orthogonally convex pentominoes.}\label{fig_pento}
\end{figure}
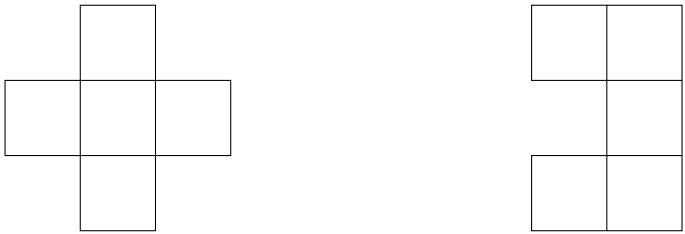

The main contribution of this paper is the following result.

\begin{Theorem}[Undecidability with Orthogonally Convex Tiles]\label{thm_main}
    Translational tiling of the plane with a set of $7$ orthogonally convex polyominoes is undecidable.
\end{Theorem}

We will prove Theorem \ref{thm_main} by reduction from Wang's domino problem (Theorem \ref{thm_berger}). The overall reduction follows Ollinger's framework \cite{o09} in spirit. Ollinger developed a framework to prove the undecidability of the translational tiling problem with $n=2$ and $k=11$, which is the first result of the undecidability concerning a fixed number of tiles. Many new techniques are incorporated into Ollinger's framework to improve Ollinger's result from parameters $(n,k)=(2,11)$ to $(n,k)=(2,7)$ by Yang and Zhang \cite{yang23,yang23b,yz24,yz24e}, and to $(n,k)=(2,5)$ by Kim \cite{k25}. To deal with orthogonally convex polyominoes, we introduce more novel techniques in this paper. One of the most important new techniques is the even higher level of abstraction of the basic building blocks for constructing the orthogonally convex polyominoes. Intuitively speaking, each high-level building block encodes a binary string of information, compared to previous works where a building block encodes only roughly one bit of information. The higher level of encoding method enables more versatile simulating ability yet retains a relatively simple structure in constructing the set of polyominoes.

The rest of the paper is organized as follows. Section~\ref{sec_bb} defines the basic building blocks, Section~\ref{sec_tileset} describes the complete set of $7$ orthogonally convex polyominoes using the building blocks, Section~\ref{sec_pattern} proves the main result (Theorem \ref{thm_main}), and Section~\ref{sec_conc} concludes with a few remarks.

\section{Building Blocks of Polyominoes}\label{sec_bb}

\subsection{Three Levels of Building Blocks} We distinguish three levels of building blocks: the normal $1\times 1$ unit squares (level-$1$ squares), level-$2$ squares with order $a$ and level-$3$ squares with order $(a,b)$, where the parameters $a$ and $b$ are positive integers.

\begin{figure}[H]
\begin{center}
\begin{tikzpicture}[scale=0.3]

\foreach \x in {-8}
\foreach \y in {0}
{
\draw [fill=gray!20] (\x+0,\y+0)--(\x,\y+1)--(\x+1,\y+1)--(\x+1,\y)--(\x+0,\y+0);
}

\foreach \x in {0,12}
\foreach \y in {0}
{
\draw [fill=gray!20] (\x+0,\y+0)--(\x,\y+1)--(\x+1,\y+1)--(\x+1,\y)--(\x+0,\y+0);
}
\foreach \x in {1,11}
\foreach \y in {-1,...,1}
{
\draw [fill=gray!20] (\x+0,\y+0)--(\x,\y+1)--(\x+1,\y+1)--(\x+1,\y)--(\x+0,\y+0);
}
\foreach \x in {2,10}
\foreach \y in {-2,...,2}
{
\draw [fill=gray!20] (\x+0,\y+0)--(\x,\y+1)--(\x+1,\y+1)--(\x+1,\y)--(\x+0,\y+0);
}
\foreach \x in {3,9}
\foreach \y in {-3,...,3}
{
\draw [fill=gray!20] (\x+0,\y+0)--(\x,\y+1)--(\x+1,\y+1)--(\x+1,\y)--(\x+0,\y+0);
}
\foreach \x in {4,8}
\foreach \y in {-4,...,4}
{
\draw [fill=gray!20] (\x+0,\y+0)--(\x,\y+1)--(\x+1,\y+1)--(\x+1,\y)--(\x+0,\y+0);
}
\foreach \x in {5,7}
\foreach \y in {-5,...,5}
{
\draw [fill=gray!20] (\x+0,\y+0)--(\x,\y+1)--(\x+1,\y+1)--(\x+1,\y)--(\x+0,\y+0);
}
\foreach \x in {6}
\foreach \y in {-6,...,6}
{
\draw [fill=gray!20] (\x+0,\y+0)--(\x,\y+1)--(\x+1,\y+1)--(\x+1,\y)--(\x+0,\y+0);
}


\filldraw[red] (6.5,0.5) circle (.2);

\end{tikzpicture}
\end{center}
\caption{A level-$1$ square (left), and a level-$2$ square of order $7$ (right).}\label{fig_s12}
\end{figure}
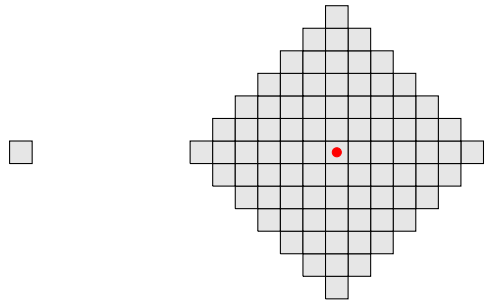

As illustrated on the left of Figure \ref{fig_s12}, the level-$1$ square is just a $1\times 1$ unit square. A level-$2$ square of order $7$ is illustrated on the right of Figure \ref{fig_s12}. In fact, a level-$2$ square is not a square, but it simulates a square because it can tile the plane with a lattice structure like a real square. In general, for any positive integer $a$, a level-$2$ square of order $a$ is defined by any translation of the following set
$$\{(x,y) \big| |x|+|y|<a\},$$
where the square at $(0,0)$ is called the \textit{center} of the level-$2$ square (see in red dot in Figure \ref{fig_s12}). By definition, a level-$2$ square of order 1 is the same as a level-$1$ square.

A level-$3$ square of order $(7,3)$ is illustrated in Figure \ref{fig_s3}. Just like a level-$2$ square, a level-$3$ square is not a square but simulates a square. It consists of $3\times 3=9$ level-$2$ squares of order $7$. The $9$ level-$2$ squares are placed together in a lattice pattern without gaps or overlaps, where their centers are defined by the set
$$\{x\mathbf{i}+y\mathbf{j}\big | \mathbf{i}=(7,6), \mathbf{j}=(-6,7), 0\leq x,y\leq 2\}.$$
In general, a level-$3$ square of order $(a,b)$ is the union of $b^2$ level-$2$ squares of order $a$, where the centers of the level-$2$ squares are defined by the set
$$\{x\mathbf{i}+y\mathbf{j}\big | \mathbf{i}=(a,a-1), \mathbf{j}=(-a+1,a), 0\leq x,y\leq b\}.$$
The unique bottom most unit square of a level-3 square is referred to as the \textit{bottom corner} (see the red dot in Figure \ref{fig_s3}).


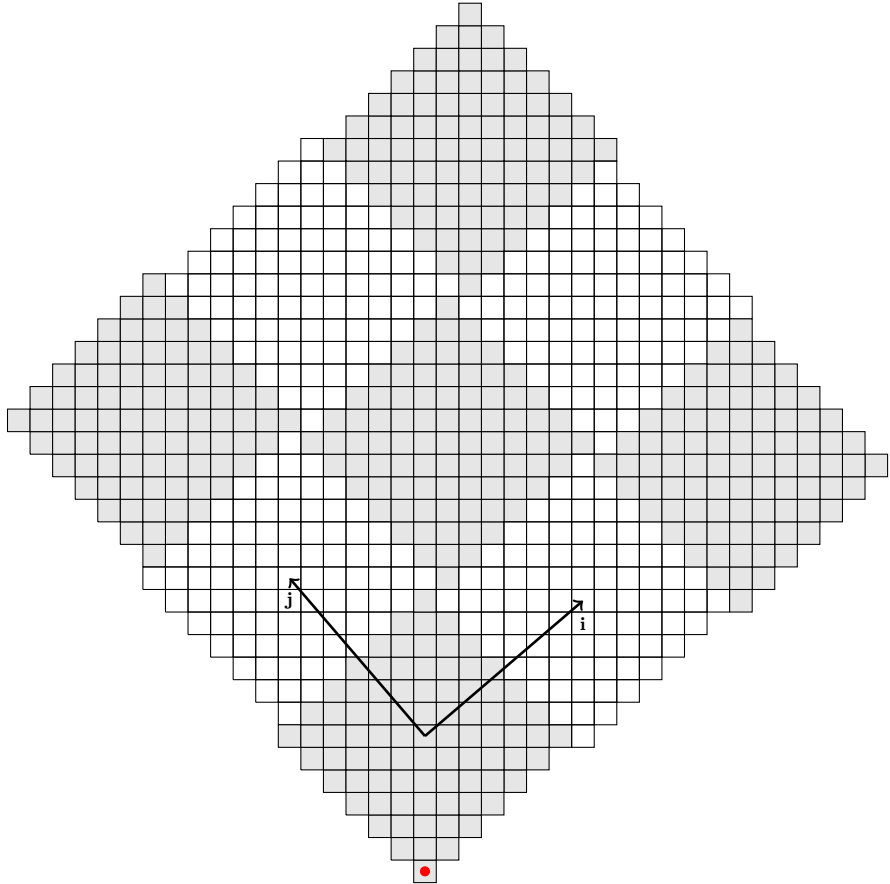
\begin{figure}[H]
\begin{center}
\begin{tikzpicture}[scale=0.3]

\foreach \x in {0,12}
\foreach \y in {0}
{
\draw [fill=gray!20] (\x+0,\y+0)--(\x,\y+1)--(\x+1,\y+1)--(\x+1,\y)--(\x+0,\y+0);
}
\foreach \x in {1,11}
\foreach \y in {-1,...,1}
{
\draw [fill=gray!20] (\x+0,\y+0)--(\x,\y+1)--(\x+1,\y+1)--(\x+1,\y)--(\x+0,\y+0);
}
\foreach \x in {2,10}
\foreach \y in {-2,...,2}
{
\draw [fill=gray!20] (\x+0,\y+0)--(\x,\y+1)--(\x+1,\y+1)--(\x+1,\y)--(\x+0,\y+0);
}
\foreach \x in {3,9}
\foreach \y in {-3,...,3}
{
\draw [fill=gray!20] (\x+0,\y+0)--(\x,\y+1)--(\x+1,\y+1)--(\x+1,\y)--(\x+0,\y+0);
}
\foreach \x in {4,8}
\foreach \y in {-4,...,4}
{
\draw [fill=gray!20] (\x+0,\y+0)--(\x,\y+1)--(\x+1,\y+1)--(\x+1,\y)--(\x+0,\y+0);
}
\foreach \x in {5,7}
\foreach \y in {-5,...,5}
{
\draw [fill=gray!20] (\x+0,\y+0)--(\x,\y+1)--(\x+1,\y+1)--(\x+1,\y)--(\x+0,\y+0);
}
\foreach \x in {6}
\foreach \y in {-6,...,6}
{
\draw [fill=gray!20] (\x+0,\y+0)--(\x,\y+1)--(\x+1,\y+1)--(\x+1,\y)--(\x+0,\y+0);
}


\foreach \x in {13,25}
\foreach \y in {-1}
{
\draw [fill=gray!20] (\x+0,\y+0)--(\x,\y+1)--(\x+1,\y+1)--(\x+1,\y)--(\x+0,\y+0);
}
\foreach \x in {14,24}
\foreach \y in {-2,...,0}
{
\draw [fill=gray!20] (\x+0,\y+0)--(\x,\y+1)--(\x+1,\y+1)--(\x+1,\y)--(\x+0,\y+0);
}
\foreach \x in {15,23}
\foreach \y in {-3,...,1}
{
\draw [fill=gray!20] (\x+0,\y+0)--(\x,\y+1)--(\x+1,\y+1)--(\x+1,\y)--(\x+0,\y+0);
}
\foreach \x in {16,22}
\foreach \y in {-4,...,2}
{
\draw [fill=gray!20] (\x+0,\y+0)--(\x,\y+1)--(\x+1,\y+1)--(\x+1,\y)--(\x+0,\y+0);
}
\foreach \x in {17,21}
\foreach \y in {-5,...,3}
{
\draw [fill=gray!20] (\x+0,\y+0)--(\x,\y+1)--(\x+1,\y+1)--(\x+1,\y)--(\x+0,\y+0);
}
\foreach \x in {18,20}
\foreach \y in {-6,...,4}
{
\draw [fill=gray!20] (\x+0,\y+0)--(\x,\y+1)--(\x+1,\y+1)--(\x+1,\y)--(\x+0,\y+0);
}
\foreach \x in {19}
\foreach \y in {-7,...,5}
{
\draw [fill=gray!20] (\x+0,\y+0)--(\x,\y+1)--(\x+1,\y+1)--(\x+1,\y)--(\x+0,\y+0);
}


\foreach \x in {7,19}
\foreach \y in {6}
{
\draw (\x+0,\y+0)--(\x,\y+1)--(\x+1,\y+1)--(\x+1,\y)--(\x+0,\y+0);
}
\foreach \x in {8,18}
\foreach \y in {5,...,7}
{
\draw (\x+0,\y+0)--(\x,\y+1)--(\x+1,\y+1)--(\x+1,\y)--(\x+0,\y+0);
}
\foreach \x in {9,17}
\foreach \y in {4,...,8}
{
\draw (\x+0,\y+0)--(\x,\y+1)--(\x+1,\y+1)--(\x+1,\y)--(\x+0,\y+0);
}
\foreach \x in {10,16}
\foreach \y in {3,...,9}
{
\draw (\x+0,\y+0)--(\x,\y+1)--(\x+1,\y+1)--(\x+1,\y)--(\x+0,\y+0);
}
\foreach \x in {11,15}
\foreach \y in {2,...,10}
{
\draw (\x+0,\y+0)--(\x,\y+1)--(\x+1,\y+1)--(\x+1,\y)--(\x+0,\y+0);
}
\foreach \x in {12,14}
\foreach \y in {1,...,11}
{
\draw (\x+0,\y+0)--(\x,\y+1)--(\x+1,\y+1)--(\x+1,\y)--(\x+0,\y+0);
}
\foreach \x in {13}
\foreach \y in {0,...,12}
{
\draw (\x+0,\y+0)--(\x,\y+1)--(\x+1,\y+1)--(\x+1,\y)--(\x+0,\y+0);
}


\foreach \x in {6,18}
\foreach \y in {-7}
{
\draw (\x+0,\y+0)--(\x,\y+1)--(\x+1,\y+1)--(\x+1,\y)--(\x+0,\y+0);
}
\foreach \x in {7,17}
\foreach \y in {-8,...,-6}
{
\draw (\x+0,\y+0)--(\x,\y+1)--(\x+1,\y+1)--(\x+1,\y)--(\x+0,\y+0);
}
\foreach \x in {8,16}
\foreach \y in {-9,...,-5}
{
\draw (\x+0,\y+0)--(\x,\y+1)--(\x+1,\y+1)--(\x+1,\y)--(\x+0,\y+0);
}
\foreach \x in {9,15}
\foreach \y in {-10,...,-4}
{
\draw (\x+0,\y+0)--(\x,\y+1)--(\x+1,\y+1)--(\x+1,\y)--(\x+0,\y+0);
}
\foreach \x in {10,14}
\foreach \y in {-11,...,-3}
{
\draw (\x+0,\y+0)--(\x,\y+1)--(\x+1,\y+1)--(\x+1,\y)--(\x+0,\y+0);
}
\foreach \x in {11,13}
\foreach \y in {-12,...,-2}
{
\draw (\x+0,\y+0)--(\x,\y+1)--(\x+1,\y+1)--(\x+1,\y)--(\x+0,\y+0);
}
\foreach \x in {12}
\foreach \y in {-13,...,-1}
{
\draw (\x+0,\y+0)--(\x,\y+1)--(\x+1,\y+1)--(\x+1,\y)--(\x+0,\y+0);
}


\foreach \x in {14,26}
\foreach \y in {12}
{
\draw [fill=gray!20] (\x+0,\y+0)--(\x,\y+1)--(\x+1,\y+1)--(\x+1,\y)--(\x+0,\y+0);
}
\foreach \x in {15,25}
\foreach \y in {11,...,13}
{
\draw [fill=gray!20] (\x+0,\y+0)--(\x,\y+1)--(\x+1,\y+1)--(\x+1,\y)--(\x+0,\y+0);
}
\foreach \x in {16,24}
\foreach \y in {10,...,14}
{
\draw [fill=gray!20] (\x+0,\y+0)--(\x,\y+1)--(\x+1,\y+1)--(\x+1,\y)--(\x+0,\y+0);
}
\foreach \x in {17,23}
\foreach \y in {9,...,15}
{
\draw [fill=gray!20] (\x+0,\y+0)--(\x,\y+1)--(\x+1,\y+1)--(\x+1,\y)--(\x+0,\y+0);
}
\foreach \x in {18,22}
\foreach \y in {8,...,16}
{
\draw [fill=gray!20] (\x+0,\y+0)--(\x,\y+1)--(\x+1,\y+1)--(\x+1,\y)--(\x+0,\y+0);
}
\foreach \x in {19,21}
\foreach \y in {7,...,17}
{
\draw [fill=gray!20] (\x+0,\y+0)--(\x,\y+1)--(\x+1,\y+1)--(\x+1,\y)--(\x+0,\y+0);
}
\foreach \x in {20}
\foreach \y in {6,...,18}
{
\draw [fill=gray!20] (\x+0,\y+0)--(\x,\y+1)--(\x+1,\y+1)--(\x+1,\y)--(\x+0,\y+0);
}


\foreach \x in {12,24}
\foreach \y in {-14}
{
\draw [fill=gray!20] (\x+0,\y+0)--(\x,\y+1)--(\x+1,\y+1)--(\x+1,\y)--(\x+0,\y+0);
}
\foreach \x in {13,23}
\foreach \y in {-15,...,-13}
{
\draw [fill=gray!20] (\x+0,\y+0)--(\x,\y+1)--(\x+1,\y+1)--(\x+1,\y)--(\x+0,\y+0);
}
\foreach \x in {14,22}
\foreach \y in {-16,...,-12}
{
\draw [fill=gray!20] (\x+0,\y+0)--(\x,\y+1)--(\x+1,\y+1)--(\x+1,\y)--(\x+0,\y+0);
}
\foreach \x in {15,21}
\foreach \y in {-17,...,-11}
{
\draw [fill=gray!20] (\x+0,\y+0)--(\x,\y+1)--(\x+1,\y+1)--(\x+1,\y)--(\x+0,\y+0);
}
\foreach \x in {16,20}
\foreach \y in {-18,...,-10}
{
\draw [fill=gray!20] (\x+0,\y+0)--(\x,\y+1)--(\x+1,\y+1)--(\x+1,\y)--(\x+0,\y+0);
}
\foreach \x in {17,19}
\foreach \y in {-19,...,-9}
{
\draw [fill=gray!20] (\x+0,\y+0)--(\x,\y+1)--(\x+1,\y+1)--(\x+1,\y)--(\x+0,\y+0);
}
\foreach \x in {18}
\foreach \y in {-20,...,-8}
{
\draw [fill=gray!20] (\x+0,\y+0)--(\x,\y+1)--(\x+1,\y+1)--(\x+1,\y)--(\x+0,\y+0);
}


\foreach \x in {20,32}
\foreach \y in {5}
{
\draw (\x+0,\y+0)--(\x,\y+1)--(\x+1,\y+1)--(\x+1,\y)--(\x+0,\y+0);
}
\foreach \x in {21,31}
\foreach \y in {4,...,6}
{
\draw (\x+0,\y+0)--(\x,\y+1)--(\x+1,\y+1)--(\x+1,\y)--(\x+0,\y+0);
}
\foreach \x in {22,30}
\foreach \y in {3,...,7}
{
\draw (\x+0,\y+0)--(\x,\y+1)--(\x+1,\y+1)--(\x+1,\y)--(\x+0,\y+0);
}
\foreach \x in {23,29}
\foreach \y in {2,...,8}
{
\draw (\x+0,\y+0)--(\x,\y+1)--(\x+1,\y+1)--(\x+1,\y)--(\x+0,\y+0);
}
\foreach \x in {24,28}
\foreach \y in {1,...,9}
{
\draw (\x+0,\y+0)--(\x,\y+1)--(\x+1,\y+1)--(\x+1,\y)--(\x+0,\y+0);
}
\foreach \x in {25,27}
\foreach \y in {0,...,10}
{
\draw (\x+0,\y+0)--(\x,\y+1)--(\x+1,\y+1)--(\x+1,\y)--(\x+0,\y+0);
}
\foreach \x in {26}
\foreach \y in {-1,...,11}
{
\draw (\x+0,\y+0)--(\x,\y+1)--(\x+1,\y+1)--(\x+1,\y)--(\x+0,\y+0);
}


\foreach \x in {19,31}
\foreach \y in {-8}
{
\draw (\x+0,\y+0)--(\x,\y+1)--(\x+1,\y+1)--(\x+1,\y)--(\x+0,\y+0);
}
\foreach \x in {20,30}
\foreach \y in {-9,...,-7}
{
\draw (\x+0,\y+0)--(\x,\y+1)--(\x+1,\y+1)--(\x+1,\y)--(\x+0,\y+0);
}
\foreach \x in {21,29}
\foreach \y in {-10,...,-6}
{
\draw (\x+0,\y+0)--(\x,\y+1)--(\x+1,\y+1)--(\x+1,\y)--(\x+0,\y+0);
}
\foreach \x in {22,28}
\foreach \y in {-11,...,-5}
{
\draw (\x+0,\y+0)--(\x,\y+1)--(\x+1,\y+1)--(\x+1,\y)--(\x+0,\y+0);
}
\foreach \x in {23,27}
\foreach \y in {-12,...,-4}
{
\draw (\x+0,\y+0)--(\x,\y+1)--(\x+1,\y+1)--(\x+1,\y)--(\x+0,\y+0);
}
\foreach \x in {24,26}
\foreach \y in {-13,...,-3}
{
\draw (\x+0,\y+0)--(\x,\y+1)--(\x+1,\y+1)--(\x+1,\y)--(\x+0,\y+0);
}
\foreach \x in {25}
\foreach \y in {-14,...,-2}
{
\draw (\x+0,\y+0)--(\x,\y+1)--(\x+1,\y+1)--(\x+1,\y)--(\x+0,\y+0);
}


\foreach \x in {26,38}
\foreach \y in {-2}
{
\draw [fill=gray!20] (\x+0,\y+0)--(\x,\y+1)--(\x+1,\y+1)--(\x+1,\y)--(\x+0,\y+0);
}
\foreach \x in {27,37}
\foreach \y in {-3,...,-1}
{
\draw [fill=gray!20] (\x+0,\y+0)--(\x,\y+1)--(\x+1,\y+1)--(\x+1,\y)--(\x+0,\y+0);
}
\foreach \x in {28,36}
\foreach \y in {-4,...,0}
{
\draw [fill=gray!20] (\x+0,\y+0)--(\x,\y+1)--(\x+1,\y+1)--(\x+1,\y)--(\x+0,\y+0);
}
\foreach \x in {29,35}
\foreach \y in {-5,...,1}
{
\draw [fill=gray!20] (\x+0,\y+0)--(\x,\y+1)--(\x+1,\y+1)--(\x+1,\y)--(\x+0,\y+0);
}
\foreach \x in {30,34}
\foreach \y in {-6,...,2}
{
\draw [fill=gray!20] (\x+0,\y+0)--(\x,\y+1)--(\x+1,\y+1)--(\x+1,\y)--(\x+0,\y+0);
}
\foreach \x in {31,33}
\foreach \y in {-7,...,3}
{
\draw [fill=gray!20] (\x+0,\y+0)--(\x,\y+1)--(\x+1,\y+1)--(\x+1,\y)--(\x+0,\y+0);
}
\foreach \x in {32}
\foreach \y in {-8,...,4}
{
\draw [fill=gray!20] (\x+0,\y+0)--(\x,\y+1)--(\x+1,\y+1)--(\x+1,\y)--(\x+0,\y+0);
}


\filldraw[red] (18.5,-19.5) circle (.2);

\draw [->, line width=1pt] (18.5,-13.5)--(25.5,-7.5);
\draw [->, line width=1pt] (18.5,-13.5)--(12.5,-6.5);

\node at (25.5,-8.5) {\tiny $\textbf{i}$}; 
\node at (12.5,-7.5) {\tiny $\textbf{j}$}; 

\end{tikzpicture}
\end{center}
\caption{A level-$3$ square of order $(7,3)$.}\label{fig_s3}
\end{figure}

Finally, level-$3$ squares can be put together to form level-$3$ polyominoes. When putting two level-$3$ squares of order $(a,b)$ together, the vector that points from the bottom corner of one level-$3$ square to the bottom corner of a neighboring level-$3$ square is one of the following two vectors:
$$\textbf{I}=(ab,(a-1)b), \textbf{J}=(-(a-1)b,ab).$$
Figure \ref{fig_d3} shows an example of a level-$3$ triomino of order $(7,2)$, which is made up of three level-$3$ squares of order $(7,2)$. The bottom corners of the three level-$3$ squares are marked by red dots, and the vectors that point from the lower red dot to the upper red dots are $(14,12)$ and $(-12,14)$.


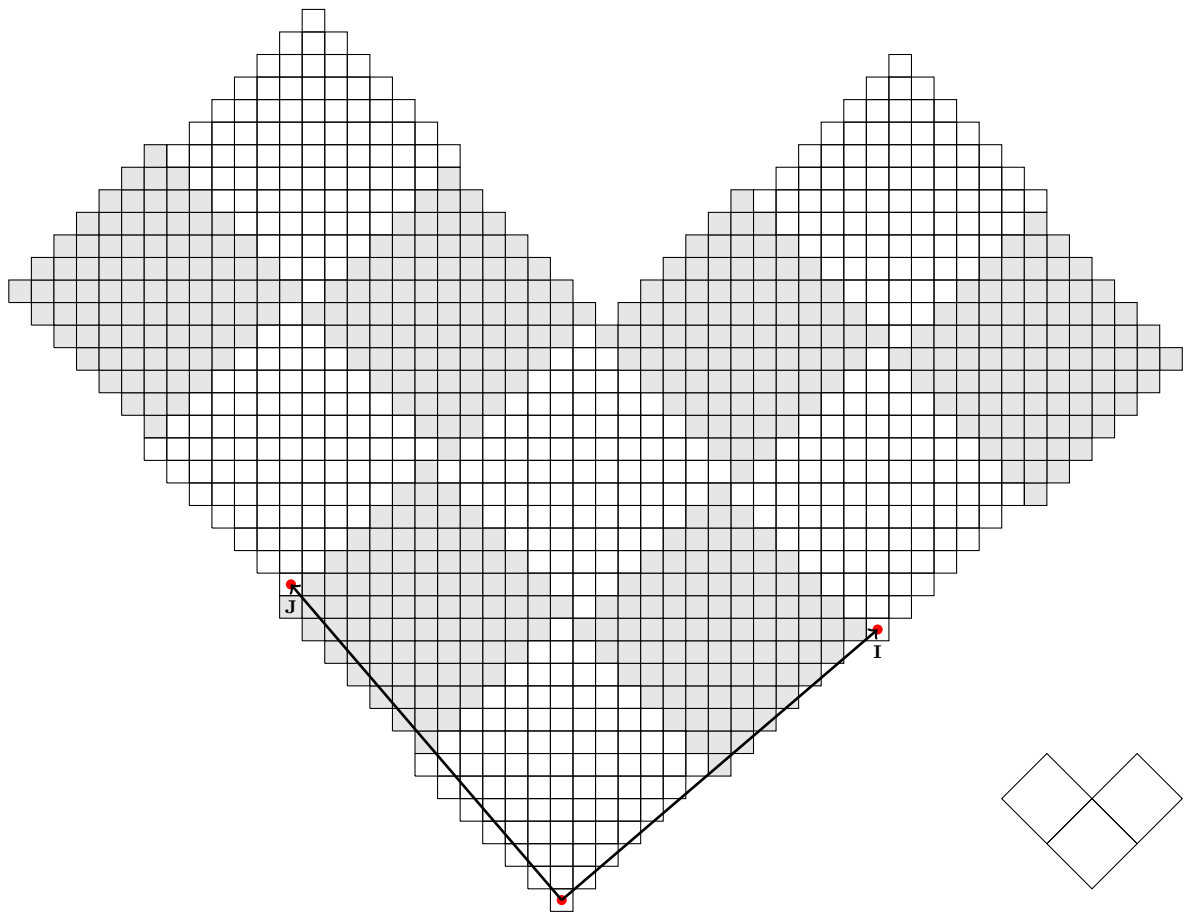
\begin{figure}[H]
\begin{center}
\begin{tikzpicture}[scale=0.3]

\foreach \x in {0,12}
\foreach \y in {0}
{
\draw [fill=gray!20] (\x+0,\y+0)--(\x,\y+1)--(\x+1,\y+1)--(\x+1,\y)--(\x+0,\y+0);
}
\foreach \x in {1,11}
\foreach \y in {-1,...,1}
{
\draw [fill=gray!20] (\x+0,\y+0)--(\x,\y+1)--(\x+1,\y+1)--(\x+1,\y)--(\x+0,\y+0);
}
\foreach \x in {2,10}
\foreach \y in {-2,...,2}
{
\draw [fill=gray!20] (\x+0,\y+0)--(\x,\y+1)--(\x+1,\y+1)--(\x+1,\y)--(\x+0,\y+0);
}
\foreach \x in {3,9}
\foreach \y in {-3,...,3}
{
\draw [fill=gray!20] (\x+0,\y+0)--(\x,\y+1)--(\x+1,\y+1)--(\x+1,\y)--(\x+0,\y+0);
}
\foreach \x in {4,8}
\foreach \y in {-4,...,4}
{
\draw [fill=gray!20] (\x+0,\y+0)--(\x,\y+1)--(\x+1,\y+1)--(\x+1,\y)--(\x+0,\y+0);
}
\foreach \x in {5,7}
\foreach \y in {-5,...,5}
{
\draw [fill=gray!20] (\x+0,\y+0)--(\x,\y+1)--(\x+1,\y+1)--(\x+1,\y)--(\x+0,\y+0);
}
\foreach \x in {6}
\foreach \y in {-6,...,6}
{
\draw [fill=gray!20] (\x+0,\y+0)--(\x,\y+1)--(\x+1,\y+1)--(\x+1,\y)--(\x+0,\y+0);
}


\foreach \x in {13,25}
\foreach \y in {-1}
{
\draw [fill=gray!20] (\x+0,\y+0)--(\x,\y+1)--(\x+1,\y+1)--(\x+1,\y)--(\x+0,\y+0);
}
\foreach \x in {14,24}
\foreach \y in {-2,...,0}
{
\draw [fill=gray!20] (\x+0,\y+0)--(\x,\y+1)--(\x+1,\y+1)--(\x+1,\y)--(\x+0,\y+0);
}
\foreach \x in {15,23}
\foreach \y in {-3,...,1}
{
\draw [fill=gray!20] (\x+0,\y+0)--(\x,\y+1)--(\x+1,\y+1)--(\x+1,\y)--(\x+0,\y+0);
}
\foreach \x in {16,22}
\foreach \y in {-4,...,2}
{
\draw [fill=gray!20] (\x+0,\y+0)--(\x,\y+1)--(\x+1,\y+1)--(\x+1,\y)--(\x+0,\y+0);
}
\foreach \x in {17,21}
\foreach \y in {-5,...,3}
{
\draw [fill=gray!20] (\x+0,\y+0)--(\x,\y+1)--(\x+1,\y+1)--(\x+1,\y)--(\x+0,\y+0);
}
\foreach \x in {18,20}
\foreach \y in {-6,...,4}
{
\draw [fill=gray!20] (\x+0,\y+0)--(\x,\y+1)--(\x+1,\y+1)--(\x+1,\y)--(\x+0,\y+0);
}
\foreach \x in {19}
\foreach \y in {-7,...,5}
{
\draw [fill=gray!20] (\x+0,\y+0)--(\x,\y+1)--(\x+1,\y+1)--(\x+1,\y)--(\x+0,\y+0);
}


\foreach \x in {7,19}
\foreach \y in {6}
{
\draw (\x+0,\y+0)--(\x,\y+1)--(\x+1,\y+1)--(\x+1,\y)--(\x+0,\y+0);
}
\foreach \x in {8,18}
\foreach \y in {5,...,7}
{
\draw (\x+0,\y+0)--(\x,\y+1)--(\x+1,\y+1)--(\x+1,\y)--(\x+0,\y+0);
}
\foreach \x in {9,17}
\foreach \y in {4,...,8}
{
\draw (\x+0,\y+0)--(\x,\y+1)--(\x+1,\y+1)--(\x+1,\y)--(\x+0,\y+0);
}
\foreach \x in {10,16}
\foreach \y in {3,...,9}
{
\draw (\x+0,\y+0)--(\x,\y+1)--(\x+1,\y+1)--(\x+1,\y)--(\x+0,\y+0);
}
\foreach \x in {11,15}
\foreach \y in {2,...,10}
{
\draw (\x+0,\y+0)--(\x,\y+1)--(\x+1,\y+1)--(\x+1,\y)--(\x+0,\y+0);
}
\foreach \x in {12,14}
\foreach \y in {1,...,11}
{
\draw (\x+0,\y+0)--(\x,\y+1)--(\x+1,\y+1)--(\x+1,\y)--(\x+0,\y+0);
}
\foreach \x in {13}
\foreach \y in {0,...,12}
{
\draw (\x+0,\y+0)--(\x,\y+1)--(\x+1,\y+1)--(\x+1,\y)--(\x+0,\y+0);
}


\foreach \x in {6,18}
\foreach \y in {-7}
{
\draw (\x+0,\y+0)--(\x,\y+1)--(\x+1,\y+1)--(\x+1,\y)--(\x+0,\y+0);
}
\foreach \x in {7,17}
\foreach \y in {-8,...,-6}
{
\draw (\x+0,\y+0)--(\x,\y+1)--(\x+1,\y+1)--(\x+1,\y)--(\x+0,\y+0);
}
\foreach \x in {8,16}
\foreach \y in {-9,...,-5}
{
\draw (\x+0,\y+0)--(\x,\y+1)--(\x+1,\y+1)--(\x+1,\y)--(\x+0,\y+0);
}
\foreach \x in {9,15}
\foreach \y in {-10,...,-4}
{
\draw (\x+0,\y+0)--(\x,\y+1)--(\x+1,\y+1)--(\x+1,\y)--(\x+0,\y+0);
}
\foreach \x in {10,14}
\foreach \y in {-11,...,-3}
{
\draw (\x+0,\y+0)--(\x,\y+1)--(\x+1,\y+1)--(\x+1,\y)--(\x+0,\y+0);
}
\foreach \x in {11,13}
\foreach \y in {-12,...,-2}
{
\draw (\x+0,\y+0)--(\x,\y+1)--(\x+1,\y+1)--(\x+1,\y)--(\x+0,\y+0);
}
\foreach \x in {12}
\foreach \y in {-13,...,-1}
{
\draw (\x+0,\y+0)--(\x,\y+1)--(\x+1,\y+1)--(\x+1,\y)--(\x+0,\y+0);
}


\foreach \x in {14,26}
\foreach \y in {12}
{
\draw [fill=gray!20] (\x+0,\y+0)--(\x,\y+1)--(\x+1,\y+1)--(\x+1,\y)--(\x+0,\y+0);
}
\foreach \x in {15,25}
\foreach \y in {11,...,13}
{
\draw [fill=gray!20] (\x+0,\y+0)--(\x,\y+1)--(\x+1,\y+1)--(\x+1,\y)--(\x+0,\y+0);
}
\foreach \x in {16,24}
\foreach \y in {10,...,14}
{
\draw [fill=gray!20] (\x+0,\y+0)--(\x,\y+1)--(\x+1,\y+1)--(\x+1,\y)--(\x+0,\y+0);
}
\foreach \x in {17,23}
\foreach \y in {9,...,15}
{
\draw [fill=gray!20] (\x+0,\y+0)--(\x,\y+1)--(\x+1,\y+1)--(\x+1,\y)--(\x+0,\y+0);
}
\foreach \x in {18,22}
\foreach \y in {8,...,16}
{
\draw [fill=gray!20] (\x+0,\y+0)--(\x,\y+1)--(\x+1,\y+1)--(\x+1,\y)--(\x+0,\y+0);
}
\foreach \x in {19,21}
\foreach \y in {7,...,17}
{
\draw [fill=gray!20] (\x+0,\y+0)--(\x,\y+1)--(\x+1,\y+1)--(\x+1,\y)--(\x+0,\y+0);
}
\foreach \x in {20}
\foreach \y in {6,...,18}
{
\draw [fill=gray!20] (\x+0,\y+0)--(\x,\y+1)--(\x+1,\y+1)--(\x+1,\y)--(\x+0,\y+0);
}


\foreach \x in {27,39}
\foreach \y in {11}
{
\draw [fill=gray!20] (\x+0,\y+0)--(\x,\y+1)--(\x+1,\y+1)--(\x+1,\y)--(\x+0,\y+0);
}
\foreach \x in {28,38}
\foreach \y in {10,...,12}
{
\draw [fill=gray!20] (\x+0,\y+0)--(\x,\y+1)--(\x+1,\y+1)--(\x+1,\y)--(\x+0,\y+0);
}
\foreach \x in {29,37}
\foreach \y in {9,...,13}
{
\draw [fill=gray!20] (\x+0,\y+0)--(\x,\y+1)--(\x+1,\y+1)--(\x+1,\y)--(\x+0,\y+0);
}
\foreach \x in {30,36}
\foreach \y in {8,...,14}
{
\draw [fill=gray!20] (\x+0,\y+0)--(\x,\y+1)--(\x+1,\y+1)--(\x+1,\y)--(\x+0,\y+0);
}
\foreach \x in {31,35}
\foreach \y in {7,...,15}
{
\draw [fill=gray!20] (\x+0,\y+0)--(\x,\y+1)--(\x+1,\y+1)--(\x+1,\y)--(\x+0,\y+0);
}
\foreach \x in {32,34}
\foreach \y in {6,...,16}
{
\draw [fill=gray!20] (\x+0,\y+0)--(\x,\y+1)--(\x+1,\y+1)--(\x+1,\y)--(\x+0,\y+0);
}
\foreach \x in {33}
\foreach \y in {5,...,17}
{
\draw [fill=gray!20] (\x+0,\y+0)--(\x,\y+1)--(\x+1,\y+1)--(\x+1,\y)--(\x+0,\y+0);
}


\foreach \x in {20,32}
\foreach \y in {5}
{
\draw (\x+0,\y+0)--(\x,\y+1)--(\x+1,\y+1)--(\x+1,\y)--(\x+0,\y+0);
}
\foreach \x in {21,31}
\foreach \y in {4,...,6}
{
\draw (\x+0,\y+0)--(\x,\y+1)--(\x+1,\y+1)--(\x+1,\y)--(\x+0,\y+0);
}
\foreach \x in {22,30}
\foreach \y in {3,...,7}
{
\draw (\x+0,\y+0)--(\x,\y+1)--(\x+1,\y+1)--(\x+1,\y)--(\x+0,\y+0);
}
\foreach \x in {23,29}
\foreach \y in {2,...,8}
{
\draw (\x+0,\y+0)--(\x,\y+1)--(\x+1,\y+1)--(\x+1,\y)--(\x+0,\y+0);
}
\foreach \x in {24,28}
\foreach \y in {1,...,9}
{
\draw (\x+0,\y+0)--(\x,\y+1)--(\x+1,\y+1)--(\x+1,\y)--(\x+0,\y+0);
}
\foreach \x in {25,27}
\foreach \y in {0,...,10}
{
\draw (\x+0,\y+0)--(\x,\y+1)--(\x+1,\y+1)--(\x+1,\y)--(\x+0,\y+0);
}
\foreach \x in {26}
\foreach \y in {-1,...,11}
{
\draw (\x+0,\y+0)--(\x,\y+1)--(\x+1,\y+1)--(\x+1,\y)--(\x+0,\y+0);
}


\foreach \x in {21,33}
\foreach \y in {18}
{
\draw (\x+0,\y+0)--(\x,\y+1)--(\x+1,\y+1)--(\x+1,\y)--(\x+0,\y+0);
}
\foreach \x in {22,32}
\foreach \y in {17,...,19}
{
\draw (\x+0,\y+0)--(\x,\y+1)--(\x+1,\y+1)--(\x+1,\y)--(\x+0,\y+0);
}
\foreach \x in {23,31}
\foreach \y in {16,...,20}
{
\draw (\x+0,\y+0)--(\x,\y+1)--(\x+1,\y+1)--(\x+1,\y)--(\x+0,\y+0);
}
\foreach \x in {24,30}
\foreach \y in {15,...,21}
{
\draw (\x+0,\y+0)--(\x,\y+1)--(\x+1,\y+1)--(\x+1,\y)--(\x+0,\y+0);
}
\foreach \x in {25,29}
\foreach \y in {14,...,22}
{
\draw (\x+0,\y+0)--(\x,\y+1)--(\x+1,\y+1)--(\x+1,\y)--(\x+0,\y+0);
}
\foreach \x in {26,28}
\foreach \y in {13,...,23}
{
\draw (\x+0,\y+0)--(\x,\y+1)--(\x+1,\y+1)--(\x+1,\y)--(\x+0,\y+0);
}
\foreach \x in {27}
\foreach \y in {12,...,24}
{
\draw (\x+0,\y+0)--(\x,\y+1)--(\x+1,\y+1)--(\x+1,\y)--(\x+0,\y+0);
}



\foreach \x in {-12,0}
\foreach \y in {14}
{
\draw [fill=gray!20] (\x+0,\y+0)--(\x,\y+1)--(\x+1,\y+1)--(\x+1,\y)--(\x+0,\y+0);
}
\foreach \x in {-11,-1}
\foreach \y in {13,...,15}
{
\draw [fill=gray!20] (\x+0,\y+0)--(\x,\y+1)--(\x+1,\y+1)--(\x+1,\y)--(\x+0,\y+0);
}
\foreach \x in {-10,-2}
\foreach \y in {12,...,16}
{
\draw [fill=gray!20] (\x+0,\y+0)--(\x,\y+1)--(\x+1,\y+1)--(\x+1,\y)--(\x+0,\y+0);
}
\foreach \x in {-9,-3}
\foreach \y in {11,...,17}
{
\draw [fill=gray!20] (\x+0,\y+0)--(\x,\y+1)--(\x+1,\y+1)--(\x+1,\y)--(\x+0,\y+0);
}
\foreach \x in {-8,-4}
\foreach \y in {10,...,18}
{
\draw [fill=gray!20] (\x+0,\y+0)--(\x,\y+1)--(\x+1,\y+1)--(\x+1,\y)--(\x+0,\y+0);
}
\foreach \x in {-7,-5}
\foreach \y in {9,...,19}
{
\draw [fill=gray!20] (\x+0,\y+0)--(\x,\y+1)--(\x+1,\y+1)--(\x+1,\y)--(\x+0,\y+0);
}
\foreach \x in {-6}
\foreach \y in {8,...,20}
{
\draw [fill=gray!20] (\x+0,\y+0)--(\x,\y+1)--(\x+1,\y+1)--(\x+1,\y)--(\x+0,\y+0);
}


\foreach \x in {1,13}
\foreach \y in {13}
{
\draw [fill=gray!20] (\x+0,\y+0)--(\x,\y+1)--(\x+1,\y+1)--(\x+1,\y)--(\x+0,\y+0);
}
\foreach \x in {2,12}
\foreach \y in {12,...,14}
{
\draw [fill=gray!20] (\x+0,\y+0)--(\x,\y+1)--(\x+1,\y+1)--(\x+1,\y)--(\x+0,\y+0);
}
\foreach \x in {3,11}
\foreach \y in {11,...,15}
{
\draw [fill=gray!20] (\x+0,\y+0)--(\x,\y+1)--(\x+1,\y+1)--(\x+1,\y)--(\x+0,\y+0);
}
\foreach \x in {4,10}
\foreach \y in {10,...,16}
{
\draw [fill=gray!20] (\x+0,\y+0)--(\x,\y+1)--(\x+1,\y+1)--(\x+1,\y)--(\x+0,\y+0);
}
\foreach \x in {5,9}
\foreach \y in {9,...,17}
{
\draw [fill=gray!20] (\x+0,\y+0)--(\x,\y+1)--(\x+1,\y+1)--(\x+1,\y)--(\x+0,\y+0);
}
\foreach \x in {6,8}
\foreach \y in {8,...,18}
{
\draw [fill=gray!20] (\x+0,\y+0)--(\x,\y+1)--(\x+1,\y+1)--(\x+1,\y)--(\x+0,\y+0);
}
\foreach \x in {7}
\foreach \y in {7,...,19}
{
\draw [fill=gray!20] (\x+0,\y+0)--(\x,\y+1)--(\x+1,\y+1)--(\x+1,\y)--(\x+0,\y+0);
}


\foreach \x in {-5,7}
\foreach \y in {20}
{
\draw (\x+0,\y+0)--(\x,\y+1)--(\x+1,\y+1)--(\x+1,\y)--(\x+0,\y+0);
}
\foreach \x in {-4,6}
\foreach \y in {19,...,21}
{
\draw (\x+0,\y+0)--(\x,\y+1)--(\x+1,\y+1)--(\x+1,\y)--(\x+0,\y+0);
}
\foreach \x in {-3,5}
\foreach \y in {18,...,22}
{
\draw (\x+0,\y+0)--(\x,\y+1)--(\x+1,\y+1)--(\x+1,\y)--(\x+0,\y+0);
}
\foreach \x in {-2,4}
\foreach \y in {17,...,23}
{
\draw (\x+0,\y+0)--(\x,\y+1)--(\x+1,\y+1)--(\x+1,\y)--(\x+0,\y+0);
}
\foreach \x in {-1,3}
\foreach \y in {16,...,24}
{
\draw (\x+0,\y+0)--(\x,\y+1)--(\x+1,\y+1)--(\x+1,\y)--(\x+0,\y+0);
}
\foreach \x in {0,2}
\foreach \y in {15,...,25}
{
\draw (\x+0,\y+0)--(\x,\y+1)--(\x+1,\y+1)--(\x+1,\y)--(\x+0,\y+0);
}
\foreach \x in {1}
\foreach \y in {14,...,26}
{
\draw (\x+0,\y+0)--(\x,\y+1)--(\x+1,\y+1)--(\x+1,\y)--(\x+0,\y+0);
}


\foreach \x in {-6,6}
\foreach \y in {7}
{
\draw (\x+0,\y+0)--(\x,\y+1)--(\x+1,\y+1)--(\x+1,\y)--(\x+0,\y+0);
}
\foreach \x in {-5,5}
\foreach \y in {6,...,8}
{
\draw (\x+0,\y+0)--(\x,\y+1)--(\x+1,\y+1)--(\x+1,\y)--(\x+0,\y+0);
}
\foreach \x in {-4,4}
\foreach \y in {5,...,9}
{
\draw (\x+0,\y+0)--(\x,\y+1)--(\x+1,\y+1)--(\x+1,\y)--(\x+0,\y+0);
}
\foreach \x in {-3,3}
\foreach \y in {4,...,10}
{
\draw (\x+0,\y+0)--(\x,\y+1)--(\x+1,\y+1)--(\x+1,\y)--(\x+0,\y+0);
}
\foreach \x in {-2,2}
\foreach \y in {3,...,11}
{
\draw (\x+0,\y+0)--(\x,\y+1)--(\x+1,\y+1)--(\x+1,\y)--(\x+0,\y+0);
}
\foreach \x in {-1,1}
\foreach \y in {2,...,12}
{
\draw (\x+0,\y+0)--(\x,\y+1)--(\x+1,\y+1)--(\x+1,\y)--(\x+0,\y+0);
}
\foreach \x in {0}
\foreach \y in {1,...,13}
{
\draw (\x+0,\y+0)--(\x,\y+1)--(\x+1,\y+1)--(\x+1,\y)--(\x+0,\y+0);
}

\draw (38,-6)--(40,-8)--(38,-10)--(36,-8)--(38,-6);
\draw (38,-10)--(36,-12)--(34,-10)--(36,-8)--(38,-10);
\draw (34,-10)--(36,-8)--(34,-6)--(32,-8)--(34,-10);


\filldraw[red] (26.5,-0.5) circle (.2);
\filldraw[red] (12.5,-12.5) circle (.2);
\filldraw[red] (0.5,1.5) circle (.2);

\draw [->,line width=1pt] (12.5,-12.5)--(26.5,-0.5);
\draw [->,line width=1pt] (12.5,-12.5)--(0.5,1.5);

\node at (26.5,-1.5) {\tiny $\textbf{I}$}; 
\node at (0.5,0.5) {\tiny $\textbf{J}$}; 

\end{tikzpicture}
\end{center}
\caption{A level-$3$ triomino consists of three level-$3$ squares of order $(7,2)$.}\label{fig_d3}
\end{figure}

\subsection{Dents and bumps of level-$2$ squares} We will use level-$2$ square of order 13 as our level-$2$ building blocks. Dents or bumps can be added to the northwest or southeast sides of these level-$2$ building blocks. The exact shape of the dent and bump is illustrated in Figure \ref{fig_db}, and the dashed lines indicate the original boundary of the level-$2$ squares before the dents or bumps are added. A level-$2$ square with a dent (resp., a bump) on the southeast side and a level-$2$ square with a bump (resp., a dent) on the northwest side can be put together perfectly without gaps or overlaps. 

However, if each of the two level-$2$ squares has a bump on their adjacent side, then the bumps will overlap. Therefore, this situation cannot happen in a complete tiling of the plane. On the other hand, if both adjacent sides have a dent, then there will be a gap left by the two level-$2$ squares. In this case, we can still form a complete tiling of the plane if we fill the gap exactly with another piece of polyomino. The \textit{tiny filler} is defined as the polyomino that can exactly fill the gap between two dents (see the polyomino in purple in Figure \ref{fig_tiny_filler}).


\begin{figure}[H]
\begin{center}
\begin{tikzpicture}[scale=0.3]

\draw [dashed] (2,1)--(2,2)--(3,2)--(3,3)--(4,3)--(4,4)--(5,4)--(5,5)--(6,5)--(6,6)--(7,6)--(7,7)--(8,7)--(8,8)--(9,8)--(9,9)--(10,9)--(10,10)--(11,10);

\draw [fill=gray!20] (-4,4)--(-3,4)--(-3,3)--(-2,3)--(-2,2)--(-1,2)--(-1,1)--(0,1)--(0,0)--(1,0)--(1,1)--(2,1)--(2,6)--(5,6)--(5,7)--(6,7)--(6,10)--(11,10)--(11,11)--(12,11)--(12,12)--(13,12)--(13,13)--(12,13)--(12,14)--(11,14)--(11,15)--(10,15)--(10,16)--(9,16)--(9,17);

\foreach \x in {20}
\foreach \y in {0}
{
\draw [fill=gray!20] (\x+-4,4+\y)--(\x+-3,4+\y)--(\x+-3,3+\y)--(\x+-2,3+\y)--(\x+-2,2+\y)--(\x+-1,2+\y)--(\x+-1,1+\y)--(\x+0,1+\y)--(\x+0,0+\y)--(\x+1,0+\y)--(\x+1,1+\y)--(\x+7,1+\y)--(\x+7,4+\y)--(\x+8,4+\y)--(\x+8,5+\y)--(\x+11,5+\y)--(\x+11,11+\y)--(\x+12,11+\y)--(\x+12,12+\y)--(\x+13,12+\y)--(\x+13,13+\y)--(\x+12,13+\y)--(\x+12,14+\y)--(\x+11,14+\y)--(\x+11,15+\y)--(\x+10,15+\y)--(\x+10,16+\y)--(\x+9,16+\y)--(\x+9,17+\y);

\draw [dashed] (\x+2,1+\y)--(\x+2,2+\y)--(\x+3,2+\y)--(\x+3,3+\y)--(\x+4,3+\y)--(\x+4,4+\y)--(\x+5,4+\y)--(\x+5,5+\y)--(\x+6,5+\y)--(\x+6,6+\y)--(\x+7,6+\y)--(\x+7,7+\y)--(\x+8,7+\y)--(\x+8,8+\y)--(\x+9,8+\y)--(\x+9,9+\y)--(\x+10,9+\y)--(\x+10,10+\y)--(\x+11,10+\y);
}

\foreach \x in {-4}
\foreach \y in {26}
{
\draw [fill=gray!20] (\x+4,-5+\y)--(\x+4,-4+\y)--(\x+3,-4+\y)--(\x+3,-3+\y)--(\x+2,-3+\y)--(\x+2,-2+\y)--(\x+1,-2+\y)--(\x+1,-1+\y)--(\x+0,-1+\y)--(\x+0,0+\y)--(\x+1,0+\y)--(\x+1,1+\y)--(\x+7,1+\y)--(\x+7,4+\y)--(\x+8,4+\y)--(\x+8,5+\y)--(\x+11,5+\y)--(\x+11,11+\y)--(\x+12,11+\y)--(\x+12,12+\y)--(\x+13,12+\y)--(\x+13,11+\y)--(\x+14,11+\y)--(\x+14,10+\y)--(\x+15,10+\y)--(\x+15,9+\y)--(\x+16,9+\y)--(\x+16,8+\y)--(\x+17,8+\y);

\draw [dashed] (\x+2,1+\y)--(\x+2,2+\y)--(\x+3,2+\y)--(\x+3,3+\y)--(\x+4,3+\y)--(\x+4,4+\y)--(\x+5,4+\y)--(\x+5,5+\y)--(\x+6,5+\y)--(\x+6,6+\y)--(\x+7,6+\y)--(\x+7,7+\y)--(\x+8,7+\y)--(\x+8,8+\y)--(\x+9,8+\y)--(\x+9,9+\y)--(\x+10,9+\y)--(\x+10,10+\y)--(\x+11,10+\y);
}

\foreach \x in {16}
\foreach \y in {26}
{
\draw [fill=gray!20] (\x+4,-5+\y)--(\x+4,-4+\y)--(\x+3,-4+\y)--(\x+3,-3+\y)--(\x+2,-3+\y)--(\x+2,-2+\y)--(\x+1,-2+\y)--(\x+1,-1+\y)--(\x+0,-1+\y)--(\x+0,0+\y)--(\x+1,0+\y)--(\x+1,1+\y)--(\x+2,1+\y)--(\x+2,6+\y)--(\x+5,6+\y)--(\x+5,7+\y)--(\x+6,7+\y)--(\x+6,10+\y)--(\x+11,10+\y)--(\x+11,11+\y)--(\x+12,11+\y)--(\x+12,12+\y)--(\x+13,12+\y)--(\x+13,11+\y)--(\x+14,11+\y)--(\x+14,10+\y)--(\x+15,10+\y)--(\x+15,9+\y)--(\x+16,9+\y)--(\x+16,8+\y)--(\x+17,8+\y);

\draw [dashed] (\x+2,1+\y)--(\x+2,2+\y)--(\x+3,2+\y)--(\x+3,3+\y)--(\x+4,3+\y)--(\x+4,4+\y)--(\x+5,4+\y)--(\x+5,5+\y)--(\x+6,5+\y)--(\x+6,6+\y)--(\x+7,6+\y)--(\x+7,7+\y)--(\x+8,7+\y)--(\x+8,8+\y)--(\x+9,8+\y)--(\x+9,9+\y)--(\x+10,9+\y)--(\x+10,10+\y)--(\x+11,10+\y);
}

\end{tikzpicture}
\end{center}
\caption{A dent or a bump on the northwest or southeast sides of a level-$2$ squares of order $13$.}\label{fig_db}
\end{figure}
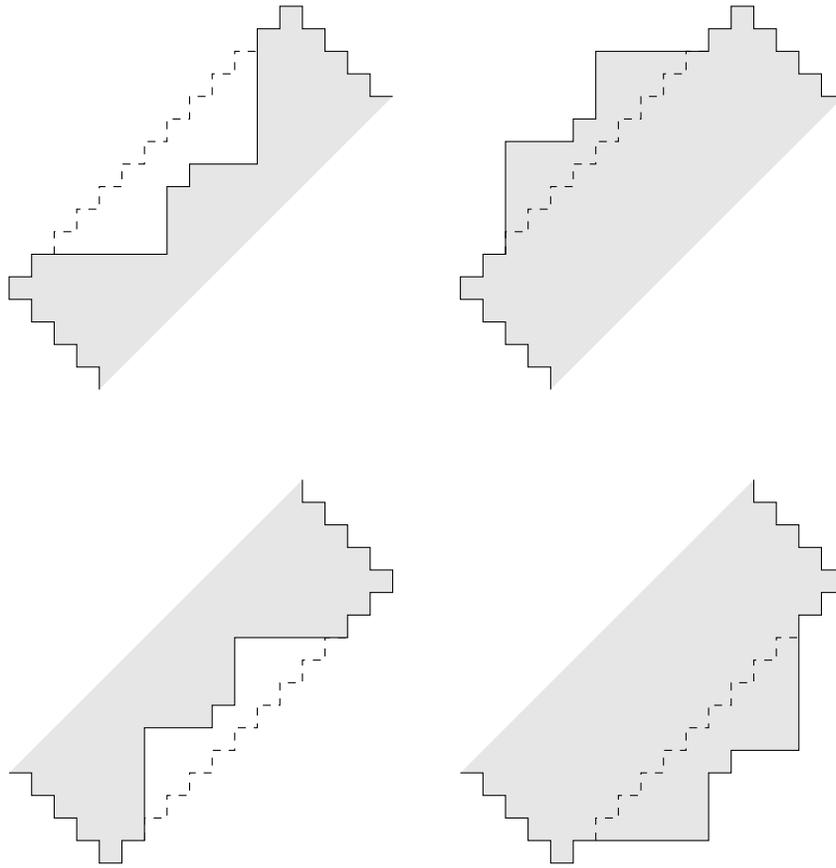


\begin{figure}[H]
\begin{center}
\begin{tikzpicture}[scale=0.3]

\draw [fill=gray!20] (-4,4)--(-3,4)--(-3,3)--(-2,3)--(-2,2)--(-1,2)--(-1,1)--(0,1)--(0,0)--(1,0)--(1,1)--(2,1)--(2,2)--(3,2)--(3,3)--(4,3)--(4,4)--(5,4)--(5,5)--(6,5)--(6,6)--(7,6)--(7,7)--(8,7)--(8,8)--(9,8)--(9,9)--(10,9)--(10,10)--(11,10)--(11,11)--(12,11)--(12,12)--(13,12)--(13,13)--(12,13)--(12,14)--(11,14)--(11,15)--(10,15)--(10,16)--(9,16)--(9,17);

\draw [fill=gray!20] (4,-5)--(4,-4)--(3,-4)--(3,-3)--(2,-3)--(2,-2)--(1,-2)--(1,-1)--(0,-1)--(0,0)--(1,0)--(1,1)--(2,1)--(2,2)--(3,2)--(3,3)--(4,3)--(4,4)--(5,4)--(5,5)--(6,5)--(6,6)--(7,6)--(7,7)--(8,7)--(8,8)--(9,8)--(9,9)--(10,9)--(10,10)--(11,10)--(11,11)--(12,11)--(12,12)--(13,12)--(13,11)--(14,11)--(14,10)--(15,10)--(15,9)--(16,9)--(16,8)--(17,8);

\draw [fill=violet!20] (2,1)--(2,6)--(5,6)--(5,7)--(6,7)--(6,10)--(11,10)--(11,5)--(8,5)--(8,4)--(7,4)--(7,1)--(2,1);
\draw [black!50] (2,2)--(3,2)--(3,3)--(4,3)--(4,4)--(5,4)--(5,5)--(6,5)--(6,6)--(7,6)--(7,7)--(8,7)--(8,8)--(9,8)--(9,9)--(10,9)--(10,10)--(11,10);

\end{tikzpicture}
\end{center}
\caption{A tiny filler between two level-$2$ squares of order $13$.}\label{fig_tiny_filler}
\end{figure}
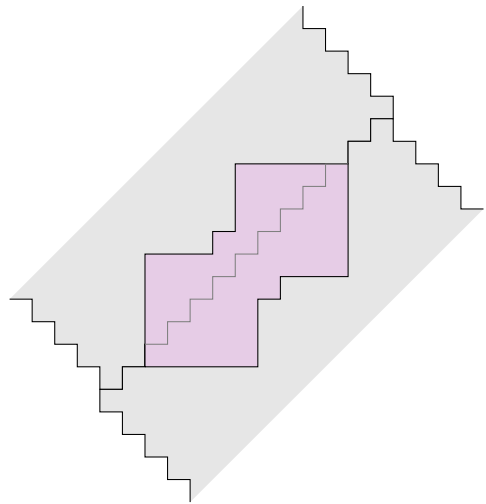

Obviously, the tiny filler is orthogonally convex. Because there exists a pseudo-hexagonal characterization for the translational tilability of a single connected tile \cite{bn91,w15}, it is easy to check that the tiny filler cannot tile the plane with translated copies of itself.

\begin{Lemma}\label{lem_tiny_1}
    The tiny filler alone cannot tile the plane by translation.
\end{Lemma}


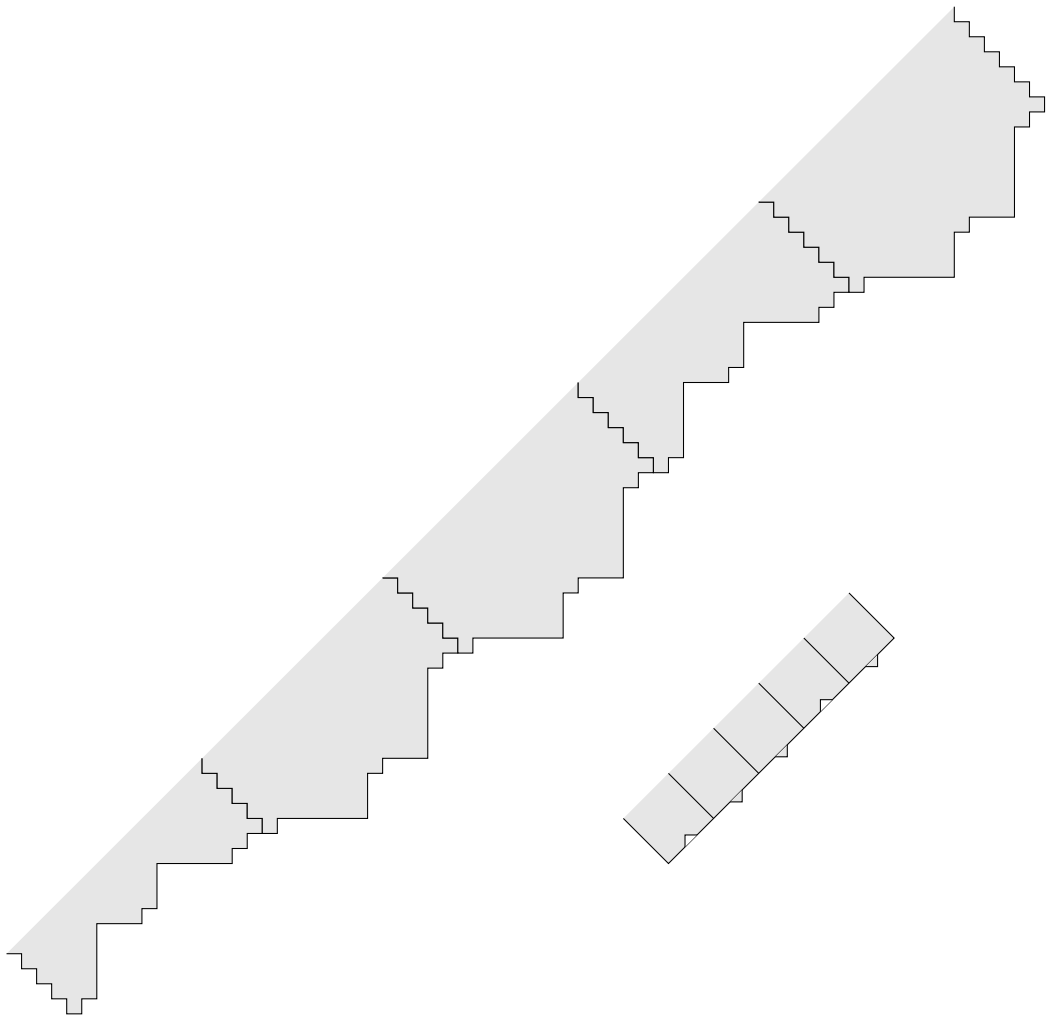
\begin{figure}[H]
\begin{center}
\begin{tikzpicture}[scale=0.2]

\draw [fill=gray!20] (-4,4)--(-3,4)--(-3,3)--(-2,3)--(-2,2)--(-1,2)--(-1,1)--(0,1)--(0,0)--(1,0)--(1,1)--(2,1)--(2,6)--(5,6)--(5,7)--(6,7)--(6,10)--(11,10)--(11,11)--(12,11)--(12,12)--(13,12)--(13,13)--(12,13)--(12,14)--(11,14)--(11,15)--(10,15)--(10,16)--(9,16)--(9,17);

\foreach \x in {13}
\foreach \y in {12}
{
\draw [fill=gray!20] (\x+-4,5+\y)--(\x+-4,4+\y)--(\x+-3,4+\y)--(\x+-3,3+\y)--(\x+-2,3+\y)--(\x+-2,2+\y)--(\x+-1,2+\y)--(\x+-1,1+\y)--(\x+0,1+\y)--(\x+0,0+\y)--(\x+1,0+\y)--(\x+1,1+\y)--(\x+7,1+\y)--(\x+7,4+\y)--(\x+8,4+\y)--(\x+8,5+\y)--(\x+11,5+\y)--(\x+11,11+\y)--(\x+12,11+\y)--(\x+12,12+\y)--(\x+13,12+\y)--(\x+13,13+\y)--(\x+12,13+\y)--(\x+12,14+\y)--(\x+11,14+\y)--(\x+11,15+\y)--(\x+10,15+\y)--(\x+10,16+\y)--(\x+9,16+\y)--(\x+9,17+\y)--(\x+8,17+\y);
}

\foreach \x in {26}
\foreach \y in {24}
{
\draw [fill=gray!20] (\x-5,5+\y)--(\x+-4,5+\y)--(\x+-4,4+\y)--(\x+-3,4+\y)--(\x+-3,3+\y)--(\x+-2,3+\y)--(\x+-2,2+\y)--(\x+-1,2+\y)--(\x+-1,1+\y)--(\x+0,1+\y)--(\x+0,0+\y)--(\x+1,0+\y)--(\x+1,1+\y)--(\x+7,1+\y)--(\x+7,4+\y)--(\x+8,4+\y)--(\x+8,5+\y)--(\x+11,5+\y)--(\x+11,11+\y)--(\x+12,11+\y)--(\x+12,12+\y)--(\x+13,12+\y)--(\x+13,13+\y)--(\x+12,13+\y)--(\x+12,14+\y)--(\x+11,14+\y)--(\x+11,15+\y)--(\x+10,15+\y)--(\x+10,16+\y)--(\x+9,16+\y)--(\x+9,17+\y)--(\x+8,17+\y)--(\x+8,18+\y);
}

\foreach \x in {39}
\foreach \y in {36}
{
\draw [fill=gray!20] (\x-5,6+\y)--(\x-5,5+\y)--(\x+-4,5+\y)--(\x+-4,4+\y)--(\x+-3,4+\y)--(\x+-3,3+\y)--(\x+-2,3+\y)--(\x+-2,2+\y)--(\x+-1,2+\y)--(\x+-1,1+\y)--(\x+0,1+\y)--(\x+0,0+\y)--(\x+1,0+\y)--(\x+1,1+\y)--(\x+2,1+\y)--(\x+2,6+\y)--(\x+5,6+\y)--(\x+5,7+\y)--(\x+6,7+\y)--(\x+6,10+\y)--(\x+11,10+\y)--(\x+11,11+\y)--(\x+12,11+\y)--(\x+12,12+\y)--(\x+13,12+\y)--(\x+13,13+\y)--(\x+12,13+\y)--(\x+12,14+\y)--(\x+11,14+\y)--(\x+11,15+\y)--(\x+10,15+\y)--(\x+10,16+\y)--(\x+9,16+\y)--(\x+9,17+\y)--(\x+8,17+\y)--(\x+8,18+\y)--(\x+7,18+\y);
}

\foreach \x in {52}
\foreach \y in {48}
{
\draw [fill=gray!20] (\x-6,6+\y)--(\x-5,6+\y)--(\x-5,5+\y)--(\x+-4,5+\y)--(\x+-4,4+\y)--(\x+-3,4+\y)--(\x+-3,3+\y)--(\x+-2,3+\y)--(\x+-2,2+\y)--(\x+-1,2+\y)--(\x+-1,1+\y)--(\x+0,1+\y)--(\x+0,0+\y)--(\x+1,0+\y)--(\x+1,1+\y)--(\x+7,1+\y)--(\x+7,4+\y)--(\x+8,4+\y)--(\x+8,5+\y)--(\x+11,5+\y)--(\x+11,11+\y)--(\x+12,11+\y)--(\x+12,12+\y)--(\x+13,12+\y)--(\x+13,13+\y)--(\x+12,13+\y)--(\x+12,14+\y)--(\x+11,14+\y)--(\x+11,15+\y)--(\x+10,15+\y)--(\x+10,16+\y)--(\x+9,16+\y)--(\x+9,17+\y)--(\x+8,17+\y)--(\x+8,18+\y)--(\x+7,18+\y)--(\x+7,19+\y);
}

\draw [fill=gray!20]  (37,13)--(40,10)--(55,25)--(52,28);
\draw  (40,16)--(43,13); \draw  (43,19)--(46,16);  \draw  (46,22)--(49,19); \draw (52,22)--(49,25);

\foreach \x in {0,3}
{
\draw [fill=white] (3*\x+41.1,3*\x+11.1)--(3*\x+41.1,3*\x+11.9)--(3*\x+41.9,3*\x+11.9);
}
\foreach \x in {1,2,4}
{
\draw [fill=gray!20] (3*\x+41.1,3*\x+11.1)--(3*\x+41.9,3*\x+11.1)--(3*\x+41.9,3*\x+11.9);
}

\end{tikzpicture}
\end{center}
\caption{Symbolic representation of dents and bumps on the side of a level-$3$ square}\label{fig_db2}
\end{figure}

Recall that a level-$3$ square with order $(13,b)$ has $b$ level-$2$ squares on each side, and a dent or a bump can be added to each of the level-$2$ squares. Figure \ref{fig_db2} illustrates the southeast side of a level-$2$ square of order $(13,5)$. A dent or a bump is added to each of the $5$ level-$2$ squares on the southeast side. We use $0$ and $1$ to denote a level-$2$ square with a dent and a bump, respectively. The sequence of five level-$2$ squares in Figure \ref{fig_db2} can be denoted by the binary string $10110$ (in clockwise order of the boundary). So, the information can be encoded as binary stings in the form of dents or bumps on the sides of level-$3$ squares. In the bottom right of Figure \ref{fig_db2}, there is also a symbolic representation of the $5$ level-$2$ squares, and we will use this symbolic representation in the next subsection.

Note that a level-$3$ square with dents or bumps is still orthogonally convex.

\subsection{Notation for the bumps and dents of level-$3$ squares} In the next section, we will use level-$3$ squares with order $(13,22)$ as our building blocks to construct the set of polyominoes that simulates a set of Wang tiles. In other words, the level-$3$ building blocks consist of $22\times 22$ level-$2$ squares, and each level-$2$ square is of order $13$. Except for the tiny filler illustrated in Figure \ref{fig_tiny_filler}, all other polyominoes that we will construct are based on level-$3$ squares of order $(13,22)$.

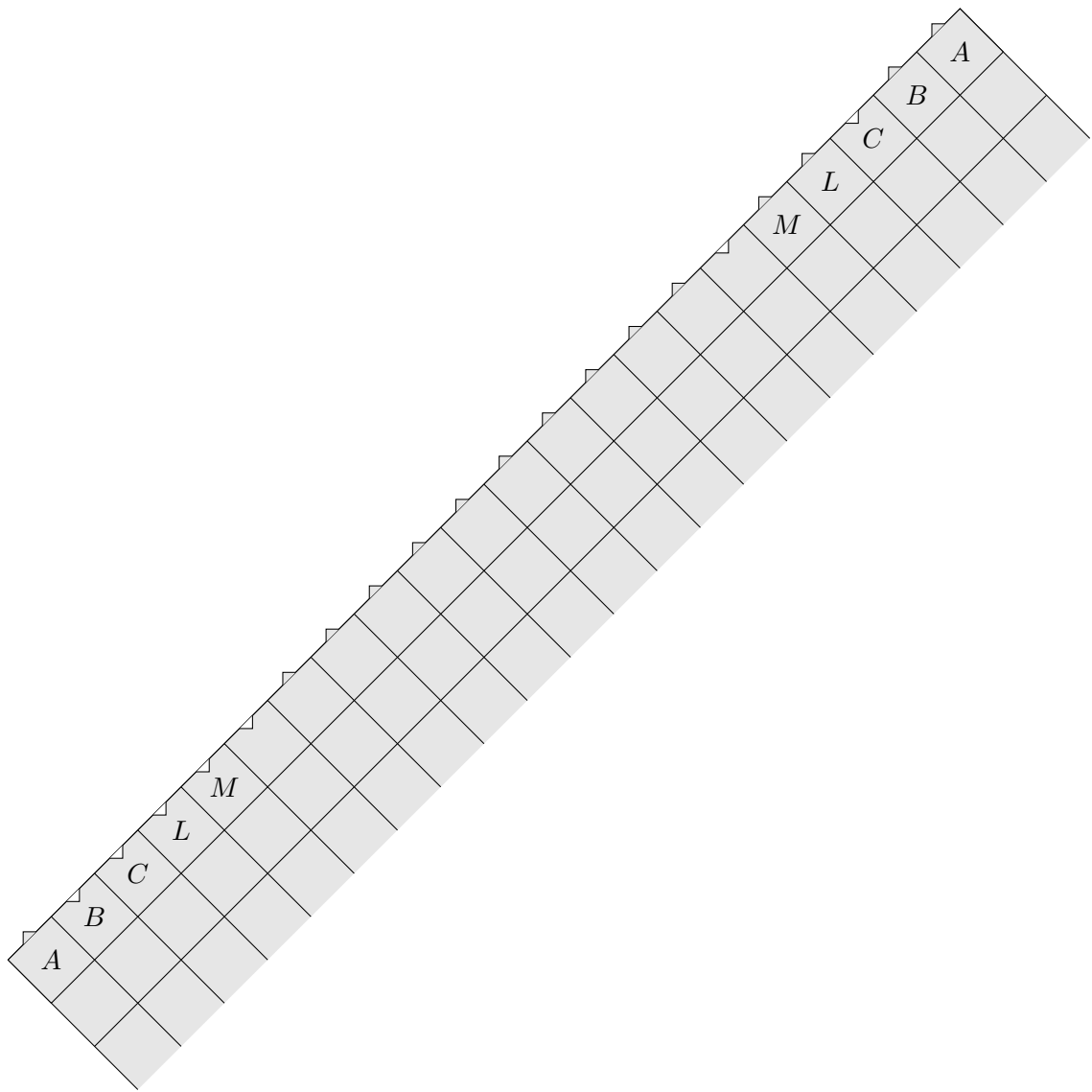
\begin{figure}[H]
\begin{center}
\begin{tikzpicture}[scale=0.3]

\draw [fill=gray!20] (6,-6)--(0,0)--(44,44)--(50,38);
\draw [fill=gray!20] (0,0)--(2,-2)--(46,42)--(44,44)--(0,0);

\foreach \x in {1,...,21}
{
\draw (2*\x,2*\x)--(2*\x+6,2*\x-6);
}
\draw (4,-4)--(48,40);

\node at (2,0) {$A$}; 
\node at (4,2) {$B$}; 
\node at (6,4) {$C$}; 
\node at (8,6) {$L$}; 
\node at (10,8) {$M$}; 

\node at (44,42) {$A$}; 
\node at (42,40) {$B$}; 
\node at (40,38) {$C$}; 
\node at (38,36) {$L$}; 
\node at (36,34) {$M$}; 

\foreach \x in {0, 6,7,8,9,10,11,12,13,14,15, 17,18,20,21}
{
\draw [fill=gray!20] (2*\x+0.7,2*\x+0.7)--(2*\x+0.7,2*\x+1.3)--(2*\x+1.3,2*\x+1.3);
}

\foreach \x in {1,2,3,4,5,16,19}
{
\draw [fill=white] (2*\x+0.7,2*\x+0.7)--(2*\x+1.3,2*\x+0.7)--(2*\x+1.3,2*\x+1.3);
}

\end{tikzpicture}
\end{center}
\caption{A level-$3$ square $\{A|C\}$ with bumps and dents on the northwest side.}\label{fig_notation1}
\end{figure}

By definition, the northwest or southeast sides of a level-$3$ have $22$ level-$2$ squares, each of which may have a bump or a dent (see Figure \ref{fig_notation1} and Figure \ref{fig_notation2} for two examples in symbolic representation). Note that in Figure \ref{fig_notation1} and Figure \ref{fig_notation2}, each gray square represents a level-$2$ square of order $13$.

\begin{figure}[H]
\begin{center}
\begin{tikzpicture}[scale=0.3]

\draw [fill=gray!20] (-4,4)--(0,0)--(44,44)--(40,48);
\draw [fill=gray!20] (0,0)--(2,-2)--(46,42)--(44,44)--(0,0);

\foreach \x in {1,...,21}
{
\draw (2*\x-4,2*\x+4)--(2*\x+2,2*\x-2);
}
\draw (-2,2)--(42,46);

\node at (2,0) {$A$}; 
\node at (4,2) {$B$}; 
\node at (6,4) {$C$}; 
\node at (8,6) {$L$}; 
\node at (10,8) {$M$}; 

\node at (44,42) {$A$}; 
\node at (42,40) {$B$}; 
\node at (40,38) {$C$}; 
\node at (38,36) {$L$}; 
\node at (36,34) {$M$}; 

\foreach \x in {3,5,16}
{
\draw [fill=gray!20] (2*\x+2.7,2*\x-1.3)--(2*\x+3.3,2*\x-1.3)--(2*\x+3.3,2*\x-0.7);
}

\foreach \x in {0,1,2,4,6,7,8,9,10,11,12,13,14,15,17,18,19,20,21}
{
\draw [fill=white] (2*\x+2.7,2*\x-1.3)--(2*\x+2.7,2*\x-0.7)--(2*\x+3.3,2*\x-0.7);
}

\end{tikzpicture}
\end{center}
\caption{A level-$3$ square $\{|A,B,C,M\}$ with bumps and dents on the southeast side.}\label{fig_notation2}
\end{figure}
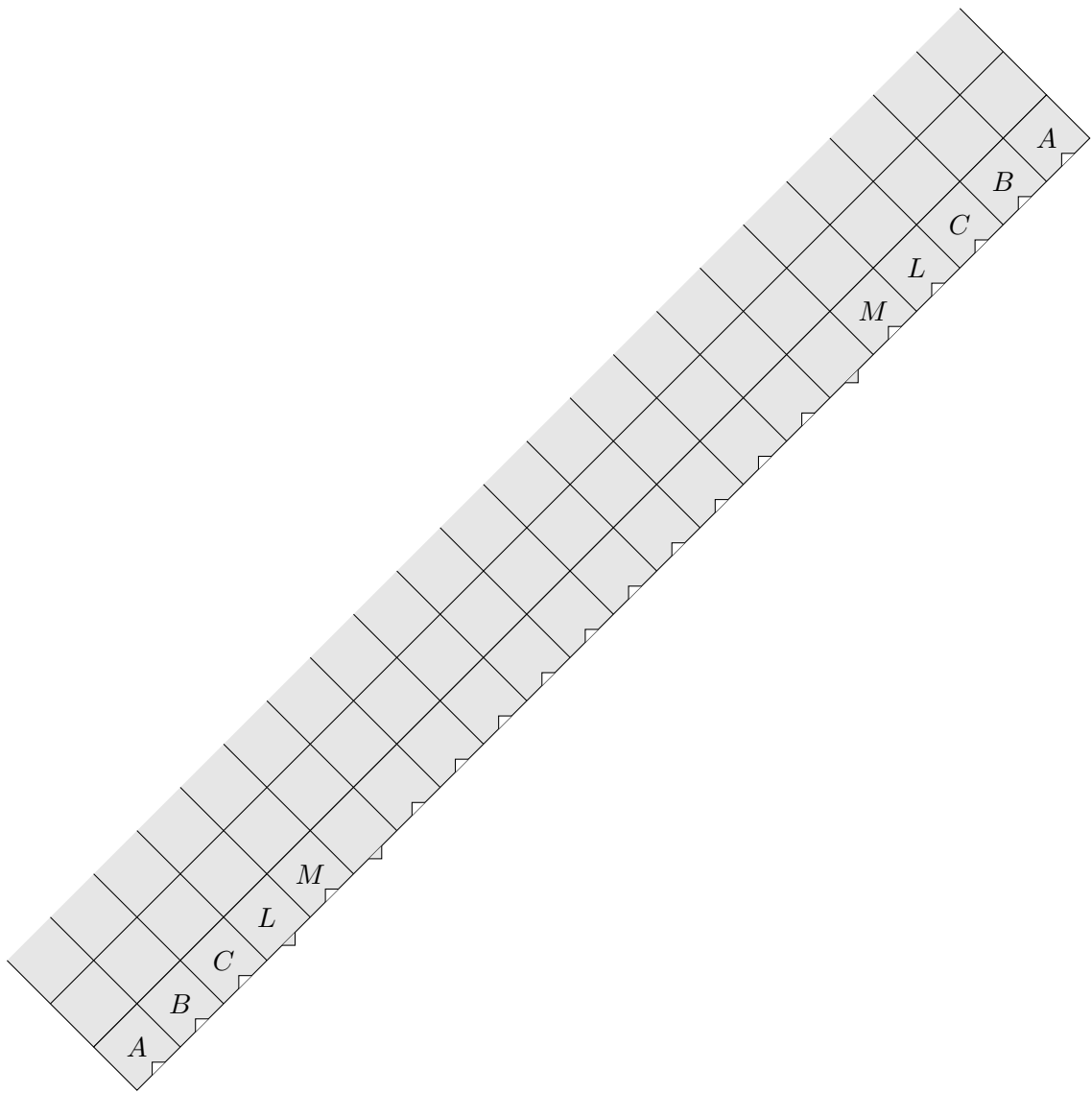

As mentioned in the previous subsection, the dents and bumps on one side of the level-$3$ square can be naturally denoted by a binary string of length $22$. However, the dents and bumps that we will actually add to the sides of the level-$3$ squares satisfy several special properties, hence we will introduce a more compact and handy notation in the next few paragraphs.

The sequence of level-$2$ squares on a northwest side or a southeast side is divided into three parts (in clockwise order): the first part consists of $5$ level-$2$ squares, the second part consists of $12$ level-$2$ squares, and the third part consists of $5$ level-$2$ squares. The $5$ squares in the first part are labeled $A$, $B$, $C$, $L$, $M$ in clockwise order, and the $5$ squares in the third part are reversely labeled as $M$, $L$, $C$, $B$, $A$ (see Figures \ref{fig_notation1} and \ref{fig_notation2}).

On the northwest side, the dents or bumps of the second part always form the sequence 011111111110, and on the southeast side, the second part is always 100000000001. This ensures that two level-$3$ polyominoes must be aligned when they are placed next to each other. In exact words, the level-$3$ squares must be matched edge-to-edge (this fact will be proved in Lemma \ref{lem_align_3} of Section \ref{sec_tileset}, after the complete set of $7$ polyominoes have been define). In particular, the second part of the southeast side of one level-$3$ square must be adjacent to the second part of the northwest side of another level-$3$ square. As a consequence, the first part (resp., third part) of the southeast side of one level-$3$ square must be adjacent to the third part (resp., first part) of another one.

To construct the set of $7$ level-$3$ polyominoes in the next section, we will use level-$3$ squares satisfying the following additional conditions: the first part contains at most four bumps (all other level-$2$ squares in the first part have dents), and the third part contains at most four dents (all other level-$2$ squares in the third parts have bumps)\footnote{These two conditions make sure that $10$ consecutive dents or $10$ consecutive bumps only appear in the second part within a level-$3$ square, and cannot appear across two adjacent level-$3$ squares on the side of a level-$3$ polyomino.}. We introduce the notation $\{X_1,...,X_i|Y_1,...,Y_j\}$ to label the dents and bumps on the northwest side or southeast side of a level-$3$ square as a whole. The notation can be interpreted as a set with two parts, a left part and a right part, separated by a vertical bar~$|$. On the left part of the set, we list all the labels of level-$2$ squares with a bump in the first part on a side of the level-$3$ square. On the right part, we list all the labels of level-$2$ squares with a dent in the third part. In other words, in this notation, we only specify the bumps of the first part and the dents of the third part. As we have mentioned earlier, the dents or bumps on the second part of a side are fixed (depending on whether it is a northwest side or a southeast side), so we do not need to list them in our notation.

We give two examples of this notation with illustrations. In Figure \ref{fig_notation1}, the northwest side is denoted by $\{A|C\}$ which means that square $A$ in the first part has a bump, and square $C$ in the third part has a dent. In Figure \ref{fig_notation2}, the southeast side is denoted by $\{|A,B,C,M\}$, because there are no bumps in the first part, and there are four dents on the southeast sides of level-$2$ squares $A$, $B$, $C$ and $M$ of the third part. 

If a level-$3$ square is labeled by $S=\{X_1,...,X_i|Y_1,...,Y_j\}$, then we use $S_L=\{X_1,...,X_i\}$ and $S_R=\{Y_1,...,Y_i\}$ to denote its left and right part, respectively. With this notation, the following lemma is straightforward. 

\begin{Lemma}\label{lem_level3}
 Let $S$ be a level-$3$ square with dents or bumps on its southeast side, and $T$ be another level-$3$ square with dents or bumps on its northwest side, then $S$ can be placed next to the northwest side of $T$ without overlap if and only if $S_L\subseteq T_R$ and $T_L\subseteq S_R$.
\end{Lemma}

Note that in Lemma \ref{lem_level3}, the conditions only guarantee the two level-$3$ squares do not overlap, but there may be some gaps left between the dents of two adjacent level-$2$ squares. This gap can be filled by the tiny filler.

\section{The Set of Seven Polyominoes}\label{sec_tileset}

Given a set of Wang tiles, we will construct a set of $7$ orthogonally convex polyominoes that simulates the set of Wang tiles. We will take the set of $3$ Wang tiles illustrated in Figure \ref{fig_wang_set} as an example to demonstrate the construction of the set of orthogonally convex polyominoes. One of them, the tiny filler has already been introduced in the previous section. The tiny filler does not depend on the set of Wang tiles, and it is always the same for any set of Wang tiles. The tiny filler is the only polyomino in the set of $7$ orthogonally convex polyominoes that is smaller than a level-$2$ building block of order $13$. All other $6$ polyominoes are level-$3$ polyominoes which are made up of level-$3$ squares. The $6$ polyominoes are: one encoder (Figure \ref{fig_encoder}), two linkers (Figure \ref{fig_linkers}) and three partial locators that always come in a group to form a locator (Figure \ref{fig_locator}).

In Figure \ref{fig_encoder}, Figure \ref{fig_linkers} and Figure \ref{fig_locator}, each gray or colored square symbolically represents a level-$3$ square of order $(13,22)$. In these Figures, a square without a label represents a normal level-$3$ square, and a square with a label represents a level-$3$ square with bumps or dents to its northwest side or southeast side on the boundary of the level-$3$ polyominoes. Because there is no level-$3$ square with both northwest and southeast sides on the outer boundary of the level-$3$ polyominoes, this labeling notation will not cause any confusion.


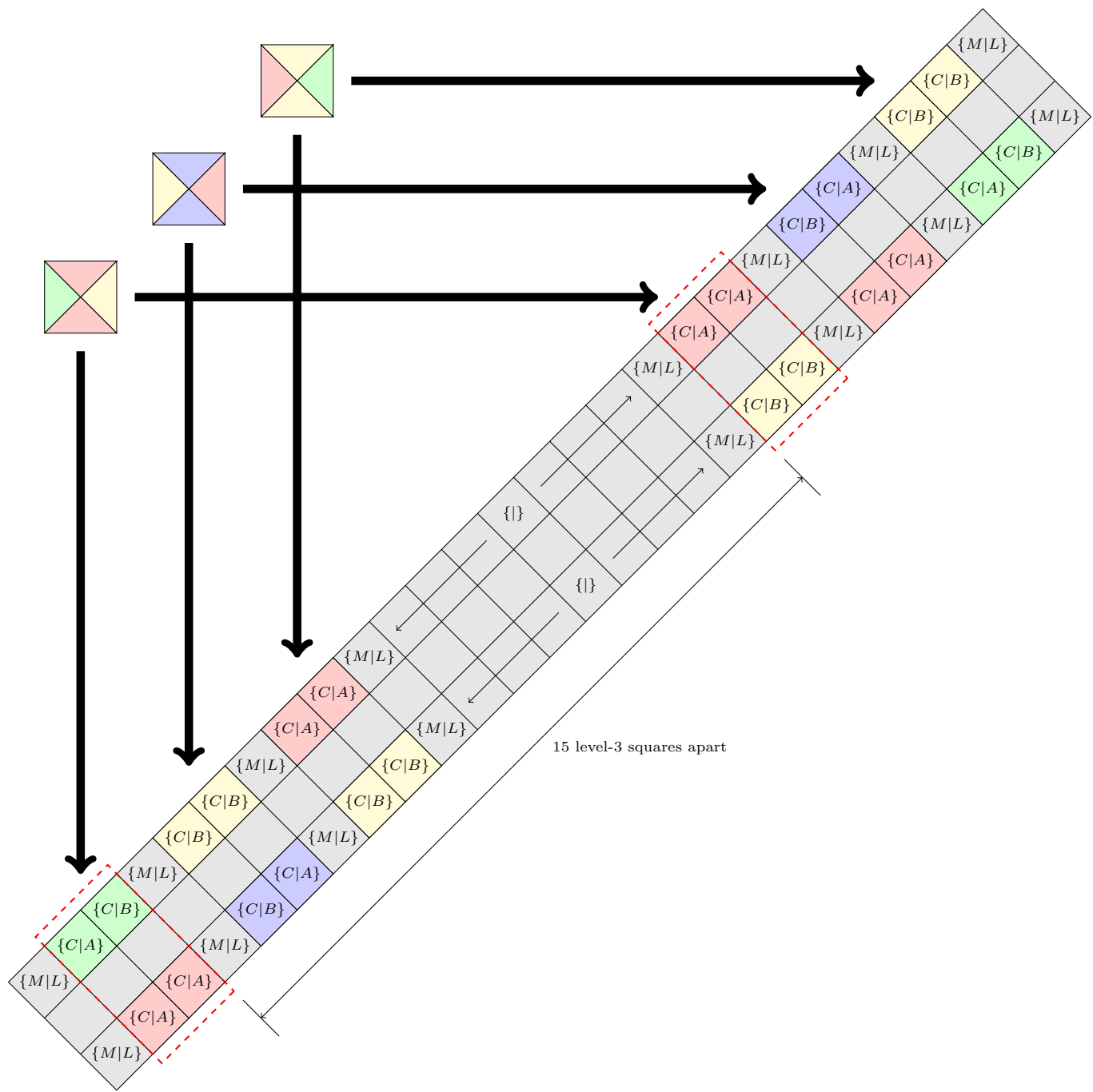
\begin{figure}[ht]
\begin{center}
\begin{tikzpicture}[scale=0.3]

\draw [ fill=gray!20] (0,0)--(54,54)--(48,60)--(-6,6)--(0,0);

\foreach \x in {0,1}
{
\draw (-2-2*\x,2*\x+2)--(-2*\x+52,2*\x+56);
}
\foreach \x in {1,...,26}
{
\draw (2*\x-6,2*\x+6)--(2*\x,2*\x);
}

\foreach \x in {1}
\foreach \y in {1}
{
\draw [ fill=red!20] (2*\x,2*\y)--(2*\x+2,2*\y+2)--(2*\x,2*\y+4)--(2*\x-2,2*\y+2)--(2*\x,2*\y);
\node at (2*\x,2*\y+2) {\tiny ${\{C|A\}}$}; 
}
\foreach \x in {2}
\foreach \y in {2}
{
\draw [ fill=red!20] (2*\x,2*\y)--(2*\x+2,2*\y+2)--(2*\x,2*\y+4)--(2*\x-2,2*\y+2)--(2*\x,2*\y);
\node at (2*\x,2*\y+2) {\tiny ${\{C|A\}}$}; 
}
\foreach \x in {21}
\foreach \y in {21}
{
\draw [ fill=red!20] (2*\x,2*\y)--(2*\x+2,2*\y+2)--(2*\x,2*\y+4)--(2*\x-2,2*\y+2)--(2*\x,2*\y);
\node at (2*\x,2*\y+2) {\tiny ${\{C|A\}}$}; 
}
\foreach \x in {22}
\foreach \y in {22}
{
\draw [ fill=red!20] (2*\x,2*\y)--(2*\x+2,2*\y+2)--(2*\x,2*\y+4)--(2*\x-2,2*\y+2)--(2*\x,2*\y);
\node at (2*\x,2*\y+2) {\tiny ${\{C|A\}}$}; 
}

\foreach \x in {4}
\foreach \y in {4}
{
\draw [ fill=blue!20] (2*\x,2*\y)--(2*\x+2,2*\y+2)--(2*\x,2*\y+4)--(2*\x-2,2*\y+2)--(2*\x,2*\y);

\node at (2*\x,2*\y+2) {\tiny ${\{C|B\}}$}; 
}
\foreach \x in {5}
\foreach \y in {5}
{
\draw [ fill=blue!20] (2*\x,2*\y)--(2*\x+2,2*\y+2)--(2*\x,2*\y+4)--(2*\x-2,2*\y+2)--(2*\x,2*\y);

\node at (2*\x,2*\y+2) {\tiny ${\{C|A\}}$}; 
}

\foreach \x in {7}
\foreach \y in {7}
{
\draw [ fill=yellow!20] (2*\x,2*\y)--(2*\x+2,2*\y+2)--(2*\x,2*\y+4)--(2*\x-2,2*\y+2)--(2*\x,2*\y);
\node at (2*\x,2*\y+2) {\tiny ${\{C|B\}}$}; 
}
\foreach \x in {8}
\foreach \y in {8}
{
\draw [ fill=yellow!20] (2*\x,2*\y)--(2*\x+2,2*\y+2)--(2*\x,2*\y+4)--(2*\x-2,2*\y+2)--(2*\x,2*\y);
\node at (2*\x,2*\y+2) {\tiny ${\{C|B\}}$}; 
}

\foreach \x in {19}
\foreach \y in {19}
{
\draw [ fill=yellow!20] (2*\x,2*\y)--(2*\x+2,2*\y+2)--(2*\x,2*\y+4)--(2*\x-2,2*\y+2)--(2*\x,2*\y);
\node at (2*\x,2*\y+2) {\tiny ${\{C|B\}}$}; 
}
\foreach \x in {18}
\foreach \y in {18}
{
\draw [ fill=yellow!20] (2*\x,2*\y)--(2*\x+2,2*\y+2)--(2*\x,2*\y+4)--(2*\x-2,2*\y+2)--(2*\x,2*\y);
\node at (2*\x,2*\y+2) {\tiny ${\{C|B\}}$}; 
}

\foreach \x in {24}
\foreach \y in {24}
{
\draw [ fill=green!20] (2*\x,2*\y)--(2*\x+2,2*\y+2)--(2*\x,2*\y+4)--(2*\x-2,2*\y+2)--(2*\x,2*\y);
\node at (2*\x,2*\y+2) {\tiny ${\{C|A\}}$}; 
}
\foreach \x in {25}
\foreach \y in {25}
{
\draw [ fill=green!20] (2*\x,2*\y)--(2*\x+2,2*\y+2)--(2*\x,2*\y+4)--(2*\x-2,2*\y+2)--(2*\x,2*\y);
\node at (2*\x,2*\y+2) {\tiny ${\{C|B\}}$}; 
}

\foreach \x in {5}
\foreach \y in {9}
{
\draw [ fill=red!20] (2*\x,2*\y)--(2*\x+2,2*\y+2)--(2*\x,2*\y+4)--(2*\x-2,2*\y+2)--(2*\x,2*\y);
\node at (2*\x,2*\y+2) {\tiny ${\{C|A\}}$}; 
}
\foreach \x in {6}
\foreach \y in {10}
{
\draw [ fill=red!20] (2*\x,2*\y)--(2*\x+2,2*\y+2)--(2*\x,2*\y+4)--(2*\x-2,2*\y+2)--(2*\x,2*\y);
\node at (2*\x,2*\y+2) {\tiny ${\{C|A\}}$}; 
}
\foreach \x in {17}
\foreach \y in {21}
{
\draw [ fill=red!20] (2*\x,2*\y)--(2*\x+2,2*\y+2)--(2*\x,2*\y+4)--(2*\x-2,2*\y+2)--(2*\x,2*\y);
\node at (2*\x,2*\y+2) {\tiny ${\{C|A\}}$}; 
}
\foreach \x in {16}
\foreach \y in {20}
{
\draw [ fill=red!20] (2*\x,2*\y)--(2*\x+2,2*\y+2)--(2*\x,2*\y+4)--(2*\x-2,2*\y+2)--(2*\x,2*\y);
\node at (2*\x,2*\y+2) {\tiny ${\{C|A\}}$}; 
}

\foreach \x in {0}
\foreach \y in {4}
{
\draw [ fill=green!20] (2*\x,2*\y)--(2*\x+2,2*\y+2)--(2*\x,2*\y+4)--(2*\x-2,2*\y+2)--(2*\x,2*\y);
\node at (2*\x,2*\y+2) {\tiny ${\{C|B\}}$}; 
}
\foreach \x in {-1}
\foreach \y in {3}
{
\draw [ fill=green!20] (2*\x,2*\y)--(2*\x+2,2*\y+2)--(2*\x,2*\y+4)--(2*\x-2,2*\y+2)--(2*\x,2*\y);
\node at (2*\x,2*\y+2) {\tiny ${\{C|A\}}$}; 
}

\foreach \x in {2}
\foreach \y in {6}
{
\draw [ fill=yellow!20] (2*\x,2*\y)--(2*\x+2,2*\y+2)--(2*\x,2*\y+4)--(2*\x-2,2*\y+2)--(2*\x,2*\y);
\node at (2*\x,2*\y+2) {\tiny ${\{C|B\}}$}; 
}
\foreach \x in {3}
\foreach \y in {7}
{
\draw [ fill=yellow!20] (2*\x,2*\y)--(2*\x+2,2*\y+2)--(2*\x,2*\y+4)--(2*\x-2,2*\y+2)--(2*\x,2*\y);
\node at (2*\x,2*\y+2) {\tiny ${\{C|B\}}$}; 
}

\foreach \x in {22}
\foreach \y in {26}
{
\draw [ fill=yellow!20] (2*\x,2*\y)--(2*\x+2,2*\y+2)--(2*\x,2*\y+4)--(2*\x-2,2*\y+2)--(2*\x,2*\y);
\node at (2*\x,2*\y+2) {\tiny ${\{C|B\}}$}; 
}
\foreach \x in {23}
\foreach \y in {27}
{
\draw [ fill=yellow!20] (2*\x,2*\y)--(2*\x+2,2*\y+2)--(2*\x,2*\y+4)--(2*\x-2,2*\y+2)--(2*\x,2*\y);
\node at (2*\x,2*\y+2) {\tiny ${\{C|B\}}$}; 
}

\foreach \x in {20}
\foreach \y in {24}
{
\draw [ fill=blue!20] (2*\x,2*\y)--(2*\x+2,2*\y+2)--(2*\x,2*\y+4)--(2*\x-2,2*\y+2)--(2*\x,2*\y);
\node at (2*\x,2*\y+2) {\tiny ${\{C|A\}}$}; 
}
\foreach \x in {19}
\foreach \y in {23}
{
\draw [ fill=blue!20] (2*\x,2*\y)--(2*\x+2,2*\y+2)--(2*\x,2*\y+4)--(2*\x-2,2*\y+2)--(2*\x,2*\y);
\node at (2*\x,2*\y+2) {\tiny ${\{C|B\}}$}; 
}

\foreach \x in {0,3,6,9,17,20,23,26}
\foreach \y in {0,4}
{
\node at (2*\x-\y,2*\x+2+\y) {\tiny ${\{M|L\}}$}; 
}

\foreach \x in {13}
\foreach \y in {0,4}
{
\draw [->] (2*\x-\y+1.5,2*\x+2+\y+1.5)--(2*\x-\y+6.5,2*\x+2+\y+6.5);
\draw [->] (2*\x-\y-1.5,2*\x+2+\y-1.5)--(2*\x-\y-6.5,2*\x+2+\y-6.5);
\node at (2*\x-\y,2*\x+2+\y) {\tiny ${\{|\}}$}; 
}

\draw [fill=green!20] (-4,42)--(-2,44)--(-4,46)--(-4,42);
\draw [fill=yellow!20] (0,42)--(-2,44)--(0,46)--(0,42);
\draw [fill=red!20] (-4,42)--(-2,44)--(0,42)--(-4,42);
\draw [fill=red!20] (-4,46)--(-2,44)--(0,46)--(-4,46);
\draw [->,line width=4pt] (1,44)--(30,44);
\draw [->,line width=4pt] (-2,41)--(-2,12);

\draw [dashed, red, thick] (6.5,5.5)--(2.5,1.5)--(-4.5,8.5)--(-0.5,12.5)--(6.5,5.5);
\draw [dashed, red, thick] (40.5,39.5)--(36.5,35.5)--(29.5,42.5)--(33.5,46.5)--(40.5,39.5);

\draw (7,5)--(9,3); \draw (37,35)--(39,33);
\draw [<->] (8,4)--(38,34);
\node at (29,19) {\tiny $15$ level-$3$ squares apart};

\foreach \x in {6}
\foreach \y in {6}
{
\draw [fill=yellow!20] (\x+-4,42+\y)--(\x+-2,44+\y)--(\x+-4,46+\y)--(\x+-4,42+\y);
\draw [fill=red!20] (\x+0,42+\y)--(\x+-2,44+\y)--(\x+0,46+\y)--(\x+0,42+\y);
\draw [fill=blue!20] (\x+-4,42+\y)--(\x+-2,44+\y)--(\x+0,42+\y)--(\x+-4,42+\y);
\draw [fill=blue!20] (\x+-4,46+\y)--(\x+-2,44+\y)--(\x+0,46+\y)--(\x+-4,46+\y);
\draw [->,line width=4pt] (\x+1,44+\y)--(\x+30,44+\y);
\draw [->,line width=4pt] (\x+-2,41+\y)--(\x+-2,12+\y);
}

\foreach \x in {12}
\foreach \y in {12}
{
\draw [fill=red!20] (\x+-4,42+\y)--(\x+-2,44+\y)--(\x+-4,46+\y)--(\x+-4,42+\y);
\draw [fill=green!20] (\x+0,42+\y)--(\x+-2,44+\y)--(\x+0,46+\y)--(\x+0,42+\y);
\draw [fill=yellow!20] (\x+-4,42+\y)--(\x+-2,44+\y)--(\x+0,42+\y)--(\x+-4,42+\y);
\draw [fill=yellow!20] (\x+-4,46+\y)--(\x+-2,44+\y)--(\x+0,46+\y)--(\x+-4,46+\y);
\draw [->,line width=4pt] (\x+1,44+\y)--(\x+30,44+\y);
\draw [->,line width=4pt] (\x+-2,41+\y)--(\x+-2,12+\y);
}

\end{tikzpicture}
\end{center}
\caption{The encoder} \label{fig_encoder}
\end{figure}

Figure \ref{fig_encoder} depicts the \textit{encoder} that simulates the set of Wang tiles in Figure \ref{fig_wang_set}. It is made up of $3\times 27$ level-$3$ squares. There are $3$ kinds of labeled level-$3$ squares in the encoder. The first kind is the level-$3$ squares with labels $\{C|A\}$ or $\{C|B\}$, and they are intended to encode the colored edges of the Wang tiles. The set of Wang tiles in Figure \ref{fig_wang_set} has $4$ different colors, so they can be encoded by binary strings with length $\log_2 4 =2$. In particular, they are encoded by two consecutive level-$3$ squares on the northwest or southeast boundary of the encoder. The colors red, green, blue and yellow are encoded by $\{C|A\}\{C|A\}$, $\{C|A\}\{C|B\}$, $\{C|B\}\{C|A\}$ and $\{C|B\}\{C|B\}$, respectively. The second kind of level-$3$ squares are the \textit{markers} $\{M|L\}$, which separate the simulated colored edges in the encoder. The third kind is labeled with $\{|\}$ for padding purpose.

Structurally, the encoder consists of three segments: the bottom left encoding segment of size $3\times 10$ (count by level-$3$ squares), the central padding segment of size $3\times 7$, and the top right encoding segment of size $3\times 10$. The four edges of each Wang tile are encoded in the encoder separately in two segments, where the bottom edge and left edge are encoded in the bottom left encoding segment, and the top edge and right edge are encoded in the top right encoding segment. In Figure \ref{fig_encoder}, the two portions that encode the first Wang tile are enclosed by the red dashed lines. The padding segment plays a role in the overall structure of the tiling. With the padding segment, the two portions corresponding to a single Wang tile are exactly $15$ level-$3$ squares apart.

In general, for a set of $n$ Wang tiles with $m$ different colors (let $t=\lceil \log_2 m \rceil$), the corresponding encoder polyomino consists of $3\times \big ( (3n-1)(t+1)+3 \big )$ level-$3$ squares. Each of the two encoding segments is of size $3\times \big ( n(t+1)+1 \big )$, and the central padding segment is of size $3\times \big ( (n-1)(t+1)+1 \big )$. The two portions that encode the same Wang tile are exactly $2(n-1)(t+1)+3$ level-$3$ squares apart.


\begin{figure}[ht]
\begin{center}
\begin{tikzpicture}[scale=0.3]

\draw [ fill=gray!20] (0,0)--(2,2)--(-4,8)--(-6,6)--(0,0);
\draw [ fill=gray!20] (8,0)--(10,2)--(4,8)--(2,6)--(8,0);

\foreach \x in {0,8}
\foreach \y in {2}
{
\draw (\x-2,\y)--(\x,\y+2);
}

\foreach \x in {-2,6}
\foreach \y in {4}
{
\draw (\x-2,\y)--(\x,\y+2);
}

\node at (0,2) {\tiny ${\{A|C\}}$}; 
\node at (-4,6) {\tiny ${\{A|C\}}$}; 

\node at (8,2) {\tiny ${\{B|C\}}$}; 
\node at (4,6) {\tiny ${\{B|C\}}$};

\end{tikzpicture}
\end{center}
\caption{The $A$-linker and $B$-linker} \label{fig_linkers}
\end{figure}
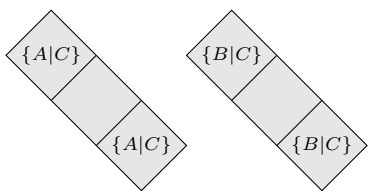


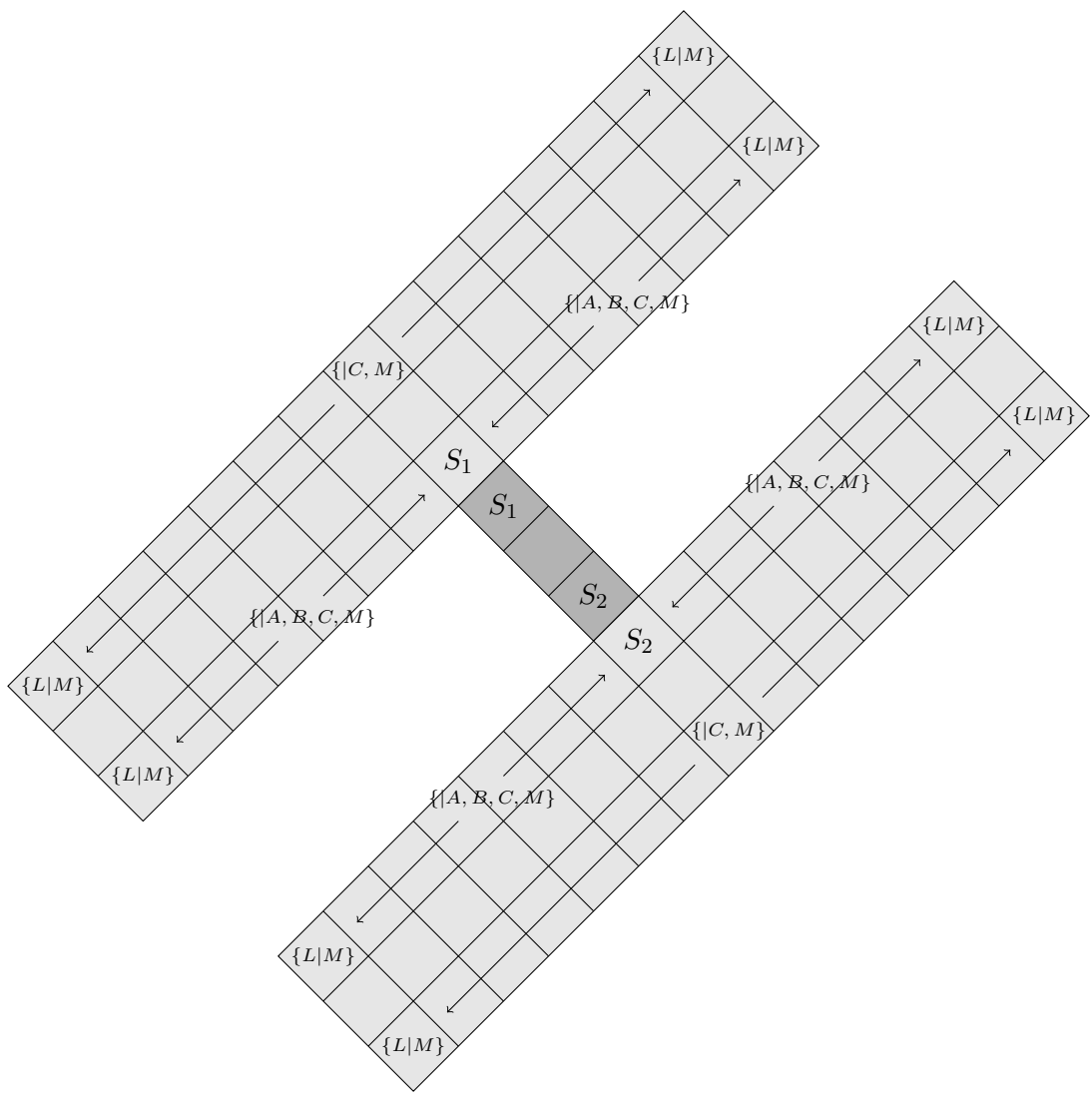
\begin{figure}[ht]
\begin{center}
\begin{tikzpicture}[scale=0.3]

\draw [ fill=gray!20] (0,0)--(30,30)--(24,36)--(-6,6)--(0,0);
\draw [ fill=gray!60] (14,14)--(20,8)--(22,10)--(16,16)--(14,14);
\draw [ fill=gray!20] (20,8)--(6,-6)--(12,-12)--(42,18)--(36,24)--(20,8);

\draw (16,12)--(18,14);\draw (18,10)--(20,12);

\foreach \x in {0,1,-5,-6}
{
\draw (-2-2*\x,2*\x+2)--(-2*\x+28,2*\x+32);
}
\foreach \x in {1,...,14}
{
\draw (2*\x-6,2*\x+6)--(2*\x,2*\x);
\draw (2*\x+6,2*\x-6)--(2*\x+12,2*\x-12);
}

\foreach \x in {0,2,-4,-6}
\foreach \y in {0,28}
{
\node at (-2*\x+\y,2*\x+2+\y) {\tiny ${\{L|M\}}$}; 
}

\foreach \x in {0,8}
{
\draw [->] (11.5+2*\x,21.5-2*\x)--(22.5+2*\x,32.5-2*\x);
\draw [->] (8.5+2*\x,18.5-2*\x)--(-2.5+2*\x,7.5-2*\x);
\node at (10+2*\x,20-2*\x) {\tiny ${\{|C,M\}}$}; 
}

\foreach \x in {0,4}
\foreach \y in {0,14}
{
\draw [->] (8+2*\x+\y,10-2*\x+\y)--(12.5+2*\x+\y,14.5-2*\x+\y);
\draw [->] (6+2*\x+\y,8-2*\x+\y)--(1.5+2*\x+\y,3.5-2*\x+\y);
\node at (7.5+2*\x+\y,9-2*\x+\y) {\tiny ${\{|A,B,C,M\}}$}; 
}

\foreach \x in {3,4}
\foreach \y in {7}
{
\node at (7+2*\x+\y,9-2*\x+\y) {$S_2$}; 
}
\foreach \x in {0,1}
\foreach \y in {7}
{
\node at (7+2*\x+\y,9-2*\x+\y) {$S_1$}; 
}

\end{tikzpicture}
\end{center}
\caption{The locator} \label{fig_locator}
\end{figure}

The two \textit{linkers}, $A$-linker and $B$-linker, are illustrated in Figure \ref{fig_linkers}, and they are used to connect adjacent color edges of simulated Wang tiles from two encoders. Like the tiny filler, the sizes of the linkers are fixed and do not depend on the size of the set of Wang tiles. Both linkers are made up of $3\times 1$ level-$3$ squares. The two level-$3$ squares on the two ends of the $A$-linker are labeled $\{A|C\}$, while the two ends of the $B$-linker are labeled $\{B|C\}$.

As illustrated in Figure \ref{fig_locator}, the \textit{locator} consists of three level-$3$ polyominoes. Two of them are made up of $3\times 15$ level-$3$ squares (in light gray in Figure \ref{fig_linkers}), and the remaining one is made up of $3\times 1$ level-$3$ squares (in dark gray). Their sizes correspond to the set of Wang tiles in Figure \ref{fig_wang_set}. In general, for a set of $n$ Wang tiles with $m$ colors (let $t=\lceil \log_2 m \rceil$), the two light gray polyominoes will consist of $3\times \big ( 2(n-1)(t+1)+3 \big )$ level-$3$ squares, and the dark gray polyomino remains the same size as $3\times 1$.

They are glued together by two kinds of dedicated level-$3$ squares labeled by $S_1=\{L,M|L,M\}$ and $S_2=\{A,L,M|A,L,M\}$ with special bumps and dents. They are special in the sense that the set $\{L,M\}$ is contained in both the left part and the right part of $S_1$ and $S_2$. In fact, no other level-$3$ squares $T$ in our set of $7$ polyominoes satisfies condition $\{L,M\}\subseteq T_R$. By Lemma \ref{lem_level3}, a level-$3$ square $S_1$ (resp. $S_2$) with dents and bumps on the southeast side can only be placed next to another level-$3$ square $S_1$ (resp. $S_2$) with dents and bumps on the northwest side. Therefore, the three partial locators always appear as a group in any tiling of the plane, and we will treat them as a combined locator in the rest of the paper. It is easy to check that each of the partial locators is orthogonally convex, but it is not orthogonally convex when they are glued together as one. To prove our main result, Theorem \ref{thm_main}, we have to split the locator into three pieces.

In addition to the special level-$3$ squares $S_1$ and $S_2$, there are three more kinds of level-$3$ squares in the locator. The \textit{selectors} $\{L|M\}$ will select a simulated Wang tile from the encoder in any tiling of the plane. The remaining two kinds of level-$3$ squares, $\{|C,M\}$ and $\{|A,B,C,M\}$, do not have too many special purposes and are there to make sure everything else works harmoniously.

As a whole, the locator can be viewed as a $9\times \big ( 2(n-1)(t+1)+3 \big )$ rectangle (of level-$3$ squares) with two $3\times \big ( (n-1)(t+1)+1 \big )$ rectangles being excavated from the northeast and southwest, respectively. The sizes of the two excavated rectangles are determined by the size of the set of Wang tiles such that the concave region left behind can contain at most $(n-1)$ simulated Wang tiles of an encoding segment (the bottom left segment or the top right segment) of the encoder. Therefore, at least one simulated Wang tile of each encoding segment must be exposed outside the locator. The size of the padding segment of the encoder is then determined by the size of the locator. With the padding segment, the distance between two encoding portions for a single Wang tile is set to be exactly the length of the longest side of the encoder, i.e., $2(n-1)(t+1)+3$.

We have completed introducing the set of $7$ polyominoes: one tiny filler and six level-$3$ polyominoes. It is straightforward to verify that all of them are orthogonally convex. We conclude this section with three lemmas on the properties of tiling the plane with this set of $7$ orthogonally convex polyominoes.

The following lemma can be viewed as a stronger version of Lemma \ref{lem_tiny_1} in Section \ref{sec_bb}.

\begin{Lemma}\label{lem_tiny_2}
    For any tiling of the entire plane with the set of $7$ orthogonally convex polyominoes introduced in this section, no two tiny fillers can be adjacent.
\end{Lemma}
\begin{proof}
   Suppose to the contrast that two tiny fillers are adjacent. We place the first tiny filler anywhere on the plane (see the purple one in Figure \ref{fig_tiny_fillers_2}). Because of the symmetry of the tiny filler, the second tiny filler essentially has two ways to be placed next to the first tiny filler. The first way is to touch the purple tiny filler on one of the longest sides of length $5$ (as the light gray tiny filler in Figure \ref{fig_tiny_fillers_2}), and the second way is to fill in the concave part of the purple tiny filler marked by red dots in Figure \ref{fig_tiny_fillers_2}.

In the first case, after placing the light gray tiny filler next to one of the longest side (see the left of Figure \ref{fig_tiny_fillers_2}), we cannot cover the locations marked with red dots with any of the $6$ big level-$3$ polyominoes without overlap. Therefore, we can only cover the red dots with yet another tiny filler (see the dark gray tiny filler on the right of Figure \ref{fig_tiny_fillers_2}). This causes a contradiction, as the partial tiling cannot extend to a complete tiling of the plane any more.

In the second case, if we cover the red dots with a second tiny filler, then it yields a contradiction immediately.
\end{proof}


\begin{figure}[H]
\begin{center}
\begin{tikzpicture}[scale=0.3]

\draw [fill=violet!20] (2,1)--(2,6)--(5,6)--(5,7)--(6,7)--(6,10)--(11,10)--(11,5)--(8,5)--(8,4)--(7,4)--(7,1)--(2,1);

\draw [->, line width=3pt] (20,7)--(26,7);

\foreach \x in {25}
\foreach \y in {0}
{
\draw [fill=violet!20] (\x+2,1+\y)--(\x+2,6+\y)--(\x+5,6+\y)--(\x+5,7+\y)--(\x+6,7+\y)--(\x+6,10+\y)--(\x+11,10+\y)--(\x+11,5+\y)--(\x+8,5+\y)--(\x+8,4+\y)--(\x+7,4+\y)--(\x+7,1+\y)--(\x+2,1+\y); 
}

\foreach \x in {9,34}
\foreach \y in {5}
{
\draw [fill=gray!20] (\x+2,1+\y)--(\x+2,6+\y)--(\x+5,6+\y)--(\x+5,7+\y)--(\x+6,7+\y)--(\x+6,10+\y)--(\x+11,10+\y)--(\x+11,5+\y)--(\x+8,5+\y)--(\x+8,4+\y)--(\x+7,4+\y)--(\x+7,1+\y)--(\x+2,1+\y); 
}

\foreach \x in {27}
\foreach \y in {-5}
{
\draw [fill=gray!60] (\x+2,1+\y)--(\x+2,6+\y)--(\x+5,6+\y)--(\x+5,7+\y)--(\x+6,7+\y)--(\x+6,10+\y)--(\x+11,10+\y)--(\x+11,5+\y)--(\x+8,5+\y)--(\x+8,4+\y)--(\x+7,4+\y)--(\x+7,1+\y)--(\x+2,1+\y); 
}

\filldraw[red] (8.5,4.5) circle (.2);
\filldraw[red] (9.5,4.5) circle (.2);
\filldraw[red] (10.5,4.5) circle (.2);

\end{tikzpicture}
\end{center}
\caption{Tiling with tiny fillers.}\label{fig_tiny_fillers_2}
\end{figure}
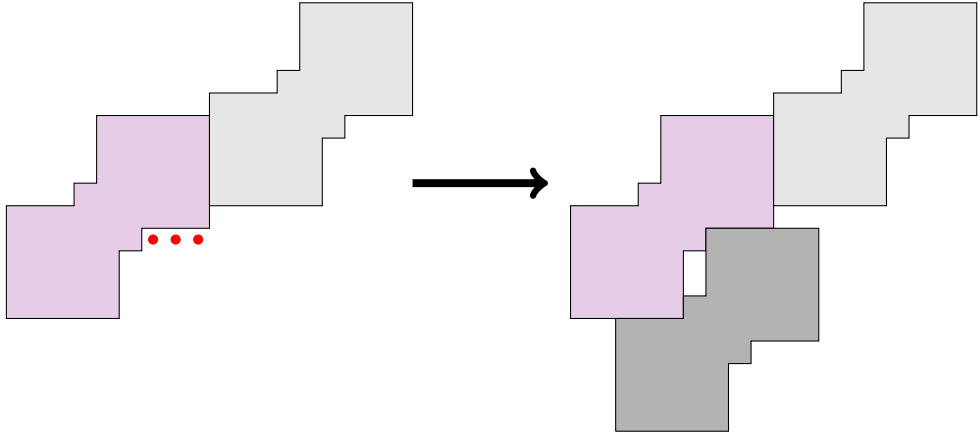

As a consequence of Lemma \ref{lem_tiny_2}, we can obtain the following lemma (Lemma \ref{lem_align_2}) on the tiling property of the level-$3$ polyominoes. Recall that each level-$3$ polyomino is made up of level-$3$ squares (of order $(13,22)$), which in turn are made up of level-$2$ squares of order $13$. Within each level-$3$ polyomino, the level-$2$ squares are aligned to a lattice form by the two vectors $\textbf{i}=(13,12)$ and $\textbf{j}=(-12,13)$. Lemma \ref{lem_align_2} claims that any tiling of the entire plane must also be aligned to this lattice.

\begin{Lemma}\label{lem_align_2}
    For any tiling of the entire plane with the set of $7$ orthogonally convex polyominoes introduced in this section, all level-$2$ squares in every level-$3$ polyominoes are aligned to a common lattice generated by the two vectors $\textbf{i}=(13,12)$ and $\textbf{j}=(-12,13)$. 
\end{Lemma}

\begin{proof}
    By Lemma \ref{lem_tiny_1}, the tiny fillers alone cannot tile the plane, so any plane tiling with the set of $7$ polyomino must make use of the level-$3$ polyominoes. By Lemma \ref{lem_tiny_2}, no two tiny fillers can be adjacent to each other in any plane tilings. Therefore, for any level-$3$ polyomino $P$ in a plane tiling, it must be surrounded by several other level-$3$ polyominoes, with possibly several isolated gaps between $P$ and the surrounding level-$3$ polyominoes such that each isolated gap (if any) can be filled exactly by a single tiny filler. 

    If two level-$2$ squares are adjacent in the northeast-southwest direction, then their boundaries must match perfectly without gaps, because there are no dents or bumps on the boundaries of northeast and southwest. If two level-$2$ squares are adjacent in the northwest-southeast direction, then dents and bumps of two adjacent level-$2$ squares must be one of the following three combinations: dent to bump, bump to dent, or dent to dent. For the first two combinations, the boundaries match perfectly without gaps, and for the third combination (i.e., dent to dent), an isolated hole can be filled by a tiny filler. In all cases, the level-$2$ squares must be aligned, this completes the proof.
\end{proof}

In addition to the alignment of level-$2$ squares, any plane tiling with the set of $7$ tiles must also be aligned with respect to level-$3$ squares. This fact will be confirmed in the proof of the main result in the next section. The following lemma states a weaker version of this fact, i.e., the alignment of level-$3$ squares in one direction (northwest-southeast).

\begin{Lemma}\label{lem_align_3}
   For any tiling of the entire plane with the set of $7$ orthogonally convex polyominoes introduced in this section, all level-$3$ squares are aligned in the northwest-southeast direction.
\end{Lemma}
\begin{proof}
This fact has in fact been mentioned in Section \ref{sec_bb} when introducing a notation to label the dents and bumps on the sides of a level-$3$ square. We are able to state it and prove it formally as a lemma only after we have completed the definition of the set of the $7$ polyominoes. By checking the dents and bumps on the northwest and southeast sides of the $6$ level-$3$ polyominoes, $10$ consecutive level-$2$ squares that all have dents (or all have bumps) appear only on the second part of the northwest or southeast sides of a level-$3$ square. By Lemma \ref{lem_align_2}, all level-$2$ squares must be aligned and matched. Therefore, the second part of one level-$3$ square must be aligned to the second part of another level-$3$ square in the northwest-southeast direction. As a consequence, the level-$3$ squares are aligned in this direction. 
\end{proof}

The alignment of level-$3$ squares in the other direction (northeast-southwest) is a consequence of the overall rigid structure of the tiling, as we will see soon in the next section.

Lemma \ref{lem_level3} and Lemma \ref{lem_align_3} will be applied repeatedly and implicitly in the next section to prove the main result. The two lemmas greatly reduce the possible way to place a tile in the process of tiling the plane.

\section{Rigid Tiling Pattern}\label{sec_pattern}

The main result (Theorem \ref{thm_main}) is proved in this section.

\begin{proof}[Proof of Theorem \ref{thm_main}]
We prove this theorem by reduction from Wang's domino problem. Given a set of Wang tiles, we have already constructed a set of $7$ polyominoes in Section \ref{sec_tileset}. To complete the proof, it suffices to show that our set of $7$ polyominoes can tile the plane if and only if the corresponding set of Wang tiles can tile the plane. To this end, we first show that if the set of $7$ polyominoes can tile the plane, then it must tile the plane in the rigid pattern illustrated in Figure \ref{fig_pattern}. In Figure \ref{fig_pattern}, the locators are shown in light and dark gray, the encoders are in orange, and the linkers are in purple.
\begin{itemize}
    \item \textbf{The locator must be used in any plane tiling.} As mentioned in the previous section, the locator is glued together from three partial locators by dedicated level-$3$ squares $S_1$ and $S_2$, so we refer to it as a single tile in the remainder of the proof. By Lemma \ref{lem_tiny_1}, at least one of the level-$3$ polyominoes must be used in plane tiling. If the encoder is used, then the locator must be used, because the marker $\{M|L\}$ on the encoder can only be matched by the selector $\{L|M\}$ in the locator. If the $A$-linker or $B$-linker is used, then either the encoder or the locator must be used, because the level-$3$ squares $\{A|C\}$ or $\{B|C\}$ have to match a level-$3$ square in the encoder or the locator. In all cases, the locator must be used.


\begin{figure}[ht]
\begin{center}
\begin{tikzpicture}[scale=0.2]


\foreach \x in {0}
\foreach \y in {0}
{
\draw [ fill=gray!20] (\x+0,\y+0)--(\x+15,\y+15)--(\x+12,\y+18)--(\x+5,\y+11)--(\x+2,\y+14)--(\x+9,\y+21)--(\x+6,\y+24)--(\x-9,\y+9)--(\x-6,\y+6)--(\x+1,\y+13)--(\x+4,\y+10)--(\x-3,\y+3)--(\x+0,\y+0);
\draw [fill=gray!60] (\x+5,\y+11)--(\x+2,\y+14)--(\x+1,\y+13)--(\x+4,\y+10)--(\x+5,\y+11);
}

\foreach \x in {12}
\foreach \y in {-12}
{
\draw [ fill=gray!20] (\x+0,\y+0)--(\x+15,\y+15)--(\x+12,\y+18)--(\x+5,\y+11)--(\x+2,\y+14)--(\x+9,\y+21)--(\x+6,\y+24)--(\x-9,\y+9)--(\x-6,\y+6)--(\x+1,\y+13)--(\x+4,\y+10)--(\x-3,\y+3)--(\x+0,\y+0);
\draw [fill=gray!60] (\x+5,\y+11)--(\x+2,\y+14)--(\x+1,\y+13)--(\x+4,\y+10)--(\x+5,\y+11);
}

\foreach \x in {34}
\foreach \y in {34}
{
\draw [ fill=gray!20] (\x+0,\y+0)--(\x+15,\y+15)--(\x+12,\y+18)--(\x+5,\y+11)--(\x+2,\y+14)--(\x+9,\y+21)--(\x+6,\y+24)--(\x-9,\y+9)--(\x-6,\y+6)--(\x+1,\y+13)--(\x+4,\y+10)--(\x-3,\y+3)--(\x+0,\y+0);
\draw [fill=gray!60] (\x+5,\y+11)--(\x+2,\y+14)--(\x+1,\y+13)--(\x+4,\y+10)--(\x+5,\y+11);
}

\foreach \x in {23}
\foreach \y in {11}
{
\draw [ fill=gray!20] (\x+0,\y+0)--(\x+15,\y+15)--(\x+12,\y+18)--(\x+5,\y+11)--(\x+2,\y+14)--(\x+9,\y+21)--(\x+6,\y+24)--(\x-9,\y+9)--(\x-6,\y+6)--(\x+1,\y+13)--(\x+4,\y+10)--(\x-3,\y+3)--(\x+0,\y+0);
\draw [fill=gray!60] (\x+5,\y+11)--(\x+2,\y+14)--(\x+1,\y+13)--(\x+4,\y+10)--(\x+5,\y+11);
}

\foreach \x in {46}
\foreach \y in {22}
{
\draw [ fill=gray!20] (\x+0,\y+0)--(\x+15,\y+15)--(\x+12,\y+18)--(\x+5,\y+11)--(\x+2,\y+14)--(\x+9,\y+21)--(\x+6,\y+24)--(\x-9,\y+9)--(\x-6,\y+6)--(\x+1,\y+13)--(\x+4,\y+10)--(\x-3,\y+3)--(\x+0,\y+0);
\draw [fill=gray!60] (\x+5,\y+11)--(\x+2,\y+14)--(\x+1,\y+13)--(\x+4,\y+10)--(\x+5,\y+11);
}

\foreach \x in {35}
\foreach \y in {-1}
{
\draw [ fill=gray!20] (\x+0,\y+0)--(\x+15,\y+15)--(\x+12,\y+18)--(\x+5,\y+11)--(\x+2,\y+14)--(\x+9,\y+21)--(\x+6,\y+24)--(\x-9,\y+9)--(\x-6,\y+6)--(\x+1,\y+13)--(\x+4,\y+10)--(\x-3,\y+3)--(\x+0,\y+0);
\draw [fill=gray!60] (\x+5,\y+11)--(\x+2,\y+14)--(\x+1,\y+13)--(\x+4,\y+10)--(\x+5,\y+11);
}

\foreach \x in {11}
\foreach \y in {23}
{
\draw [ fill=gray!20] (\x+0,\y+0)--(\x+15,\y+15)--(\x+12,\y+18)--(\x+5,\y+11)--(\x+2,\y+14)--(\x+9,\y+21)--(\x+6,\y+24)--(\x-9,\y+9)--(\x-6,\y+6)--(\x+1,\y+13)--(\x+4,\y+10)--(\x-3,\y+3)--(\x+0,\y+0);
\draw [fill=gray!60] (\x+5,\y+11)--(\x+2,\y+14)--(\x+1,\y+13)--(\x+4,\y+10)--(\x+5,\y+11);
}


\foreach \x in {5}
\foreach \y in {11}
{
\draw [ fill=orange!20] (\x+0,\y+0)--(\x+27,\y+27)--(\x+24,\y+30)--(\x-3,\y+3)--(\x,\y);
\draw [color=gray!50] (\x+1,\y+1)--(\x-2,\y+4);
\draw [color=gray!50] (\x+3,\y+3)--(\x,\y+6);
\draw [color=gray!50] (\x+4,\y+4)--(\x+1,\y+7);
\draw [color=gray!50] (\x+6,\y+6)--(\x+3,\y+9);
\draw [color=gray!50] (\x+7,\y+7)--(\x+4,\y+10);
\draw [color=gray!50] (\x+9,\y+9)--(\x+6,\y+12);
\draw [color=gray!50] (\x+10,\y+10)--(\x+7,\y+13);

\draw [color=gray!50] (\x+17,\y+17)--(\x+14,\y+20);
\draw [color=gray!50] (\x+18,\y+18)--(\x+15,\y+21);
\draw [color=gray!50] (\x+20,\y+20)--(\x+17,\y+23);
\draw [color=gray!50] (\x+21,\y+21)--(\x+18,\y+24);
\draw [color=gray!50] (\x+23,\y+23)--(\x+20,\y+26);
\draw [color=gray!50] (\x+24,\y+24)--(\x+21,\y+27);
\draw [color=gray!50] (\x+26,\y+26)--(\x+23,\y+29);
}

\foreach \x in {20}
\foreach \y in {2}
{
\draw [ fill=orange!20] (\x+0,\y+0)--(\x+27,\y+27)--(\x+24,\y+30)--(\x-3,\y+3)--(\x,\y);
\draw [color=gray!50] (\x+1,\y+1)--(\x-2,\y+4);
\draw [color=gray!50] (\x+3,\y+3)--(\x,\y+6);
\draw [color=gray!50] (\x+4,\y+4)--(\x+1,\y+7);
\draw [color=gray!50] (\x+6,\y+6)--(\x+3,\y+9);
\draw [color=gray!50] (\x+7,\y+7)--(\x+4,\y+10);
\draw [color=gray!50] (\x+9,\y+9)--(\x+6,\y+12);
\draw [color=gray!50] (\x+10,\y+10)--(\x+7,\y+13);

\draw [color=gray!50] (\x+17,\y+17)--(\x+14,\y+20);
\draw [color=gray!50] (\x+18,\y+18)--(\x+15,\y+21);
\draw [color=gray!50] (\x+20,\y+20)--(\x+17,\y+23);
\draw [color=gray!50] (\x+21,\y+21)--(\x+18,\y+24);
\draw [color=gray!50] (\x+23,\y+23)--(\x+20,\y+26);
\draw [color=gray!50] (\x+24,\y+24)--(\x+21,\y+27);
\draw [color=gray!50] (\x+26,\y+26)--(\x+23,\y+29);
}

\foreach \x in {34}
\foreach \y in {28}
{
\draw [ fill=orange!20] (\x+24,\y+30)--(\x-3,\y+3)--(\x+0,\y+0)--(\x+27,\y+27)--(\x+24,\y+30);
\draw [color=gray!50] (\x+1,\y+1)--(\x-2,\y+4);
\draw [color=gray!50] (\x+3,\y+3)--(\x,\y+6);
\draw [color=gray!50] (\x+4,\y+4)--(\x+1,\y+7);
\draw [color=gray!50] (\x+6,\y+6)--(\x+3,\y+9);
\draw [color=gray!50] (\x+7,\y+7)--(\x+4,\y+10);
\draw [color=gray!50] (\x+9,\y+9)--(\x+6,\y+12);
\draw [color=gray!50] (\x+10,\y+10)--(\x+7,\y+13);

\draw [color=gray!50] (\x+17,\y+17)--(\x+14,\y+20);
\draw [color=gray!50] (\x+18,\y+18)--(\x+15,\y+21);
\draw [color=gray!50] (\x+20,\y+20)--(\x+17,\y+23);
\draw [color=gray!50] (\x+21,\y+21)--(\x+18,\y+24);
\draw [color=gray!50] (\x+23,\y+23)--(\x+20,\y+26);
\draw [color=gray!50] (\x+24,\y+24)--(\x+21,\y+27);
\draw [color=gray!50] (\x+26,\y+26)--(\x+23,\y+29);
}

\foreach \x in {0}
\foreach \y in {-6}
{
\draw [ fill=orange!20] (\x+24,\y+30)--(\x-3,\y+3)--(\x+0,\y+0)--(\x+27,\y+27)--(\x+24,\y+30);
\draw [color=gray!50] (\x+1,\y+1)--(\x-2,\y+4);
\draw [color=gray!50] (\x+3,\y+3)--(\x,\y+6);
\draw [color=gray!50] (\x+4,\y+4)--(\x+1,\y+7);
\draw [color=gray!50] (\x+6,\y+6)--(\x+3,\y+9);
\draw [color=gray!50] (\x+7,\y+7)--(\x+4,\y+10);
\draw [color=gray!50] (\x+9,\y+9)--(\x+6,\y+12);
\draw [color=gray!50] (\x+10,\y+10)--(\x+7,\y+13);

\draw [color=gray!50] (\x+17,\y+17)--(\x+14,\y+20);
\draw [color=gray!50] (\x+18,\y+18)--(\x+15,\y+21);
\draw [color=gray!50] (\x+20,\y+20)--(\x+17,\y+23);
\draw [color=gray!50] (\x+21,\y+21)--(\x+18,\y+24);
\draw [color=gray!50] (\x+23,\y+23)--(\x+20,\y+26);
\draw [color=gray!50] (\x+24,\y+24)--(\x+21,\y+27);
\draw [color=gray!50] (\x+26,\y+26)--(\x+23,\y+29);
}


\foreach \x in {5,6,23,22}
{
\draw [fill=violet!20] (\x+5,\x+17)--(\x+2,\x+20)--(\x+1,\x+19)--(\x+4,\x+16)--(\x+5,\x+17);
}

\foreach \x in {28,...,33}
{
\draw [fill=violet!20] (\x+5,\x+11)--(\x+2,\x+14)--(\x+1,\x+13)--(\x+4,\x+10)--(\x+5,\x+11);
}

\foreach \x in {28,29,11,12}
{
\draw [fill=violet!20] (\x+5,\x+5)--(\x+2,\x+8)--(\x+1,\x+7)--(\x+4,\x+4)--(\x+5,\x+5);
}

\foreach \x in {24,...,29}
{
\draw [fill=violet!20] (\x+5,\x-1)--(\x+2,\x+2)--(\x+1,\x+1)--(\x+4,\x-2)--(\x+5,\x-1);
}

\foreach \x in {34,35,18,17}
{
\draw [fill=violet!20] (\x+5,\x-7)--(\x+2,\x-4)--(\x+1,\x-5)--(\x+4,\x-8)--(\x+5,\x-7);
}

\foreach \x in {13,14,15,43,44,45}
{
\draw [fill=violet!20] (\x+5,\x-13)--(\x+2,\x-10)--(\x+1,\x-11)--(\x+4,\x-14)--(\x+5,\x-13);
}

\foreach \x in {24,23,40,41}
{
\draw [fill=violet!20] (\x+5,\x-19)--(\x+2,\x-16)--(\x+1,\x-17)--(\x+4,\x-20)--(\x+5,\x-19);
}

\node at (24,26) {$1$};
\node at (1,16) {$2$}; \node at (13,3) {$3$};
\node at (47,37) {$4$}; \node at (35,49) {$5$};

\node at (13,10) {$a$}; \node at (47,44) {$b$};
\filldraw[black] (15,14) circle (.2);
\filldraw[black] (17,12) circle (.2);
\filldraw[black] (35,34) circle (.2);
\filldraw[black] (37,32) circle (.2);

\filldraw[red] (34,33) circle (.2);
\filldraw[red] (33,32) circle (.2);
\filldraw[red] (36,31) circle (.2);
\filldraw[red] (35,30) circle (.2);
\filldraw[red] (16,15) circle (.2);
\filldraw[red] (17,16) circle (.2);
\filldraw[red] (18,13) circle (.2);
\filldraw[red] (19,14) circle (.2);

\filldraw[red] (22,9) circle (.2);
\filldraw[red] (23,10) circle (.2);
\filldraw[red] (24,7) circle (.2);
\filldraw[red] (25,8) circle (.2);
\filldraw[red] (40,27) circle (.2);
\filldraw[red] (39,26) circle (.2);
\filldraw[red] (42,25) circle (.2);
\filldraw[red] (41,24) circle (.2);

\filldraw[red] (30,37) circle (.2);
\filldraw[red] (29,36) circle (.2);
\filldraw[red] (28,39) circle (.2);
\filldraw[red] (27,38) circle (.2);
\filldraw[red] (12,19) circle (.2);
\filldraw[red] (13,20) circle (.2);
\filldraw[red] (10,21) circle (.2);
\filldraw[red] (11,22) circle (.2);

\filldraw[red] (-1,-2) circle (.2);
\filldraw[red] (0,-1) circle (.2);
\filldraw[red] (1,-4) circle (.2);
\filldraw[red] (2,-3) circle (.2);

\filldraw[red] (51,50) circle (.2);
\filldraw[red] (50,49) circle (.2);
\filldraw[red] (53,48) circle (.2);
\filldraw[red] (52,47) circle (.2);

\draw (52,0)--(60,0)--(60,-8)--(52,-8)--(52,0);
\draw (56,0)--(56,-8);
\draw (52,-4)--(60,-4);

\end{tikzpicture}
\end{center}
\caption{The rigid tiling pattern} \label{fig_pattern}
\end{figure}
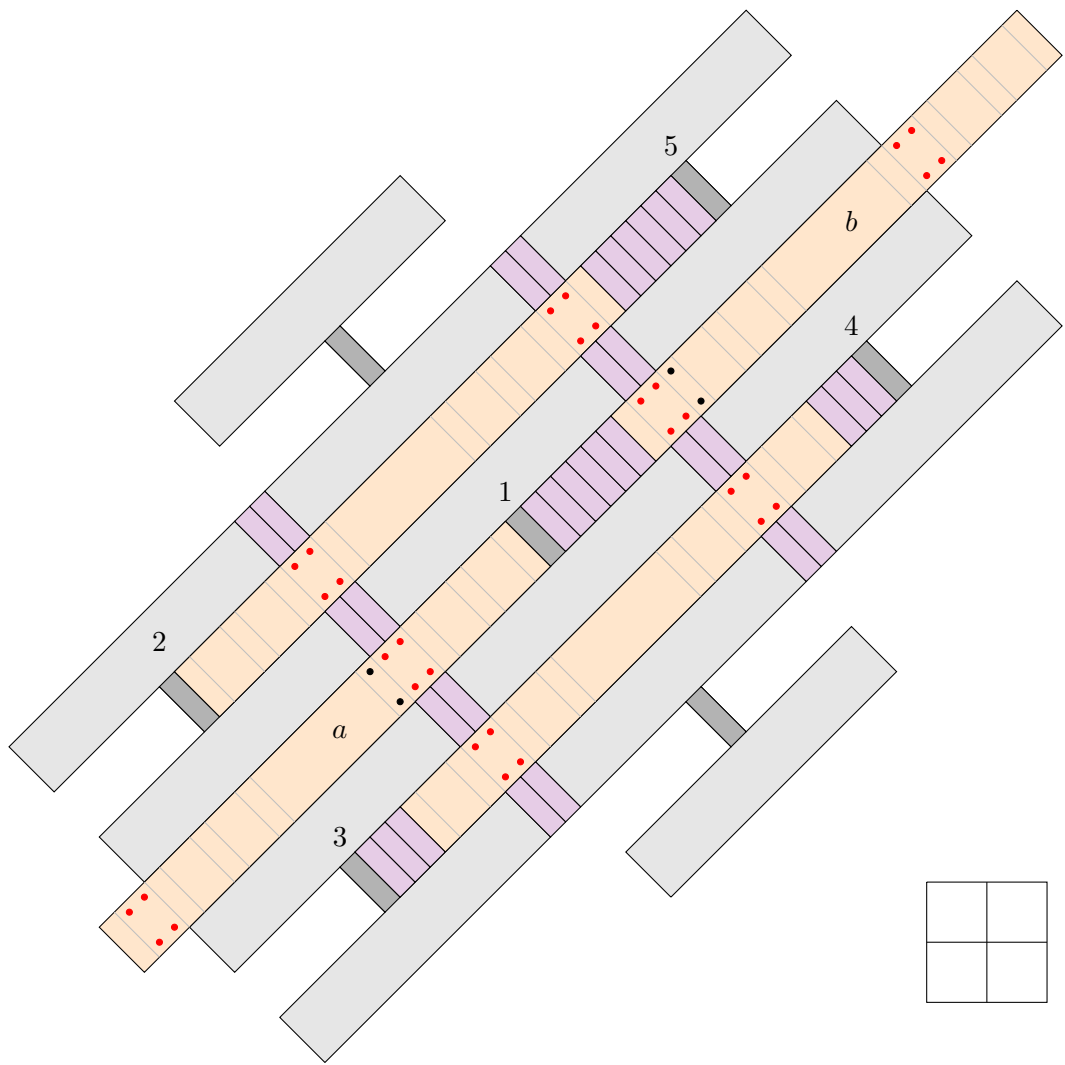

    \item \textbf{The locators form a rigid lattice structure.} Because the locators must be used, we place a locator anywhere on the plane (the locator marked $1$ in Figure \ref{fig_pattern}). Because the selectors $\{L|M\}$ in the locator can only be matched by the markers $\{M|L\}$ in the encoders, two encoders have to be placed partially inside the concave part of the locator $1$ (the encoders marked $a$ and $b$ in Figure \ref{fig_pattern}). There is flexibility of how the encoder is placed inside the locator, and we will discuss this in the next paragraph. No matter how the encoder is placed, on one end of the encoder, at most two simulated Wang tiles of one encoding segment are put inside the locator, and at least one simulated Wang tile is exposed outside. Right next to the first exposed simulated Wang tiles, there are markers (marked by the black dots in Figure \ref{fig_pattern}) exposed outside locator $1$. Again, these markers must be matched by the selectors in the locators, so the locators $2$, $3$, $4$ and $5$ must be placed as illustrated in Figure \ref{fig_pattern}. The above arguments could be applied to the newly added locators, and the pattern extends infinitely to the entire plane. Therefore, the locators form a rigid lattice structure.
    \item \textbf{Each encoder has the flexibility to expose one of the simulating Wang tiles.} As mentioned in the previous paragraph, though the locators form a rigid lattice structure, each encoder independently has the freedom to choose one of three simulated Wang tiles to be exposed outside the locators, which is marked by the red dots in Figure \ref{fig_pattern}. As we have mentioned in the previous section, the distance between two portions of a simulated Wang tile is identical to the longest side of the locator. Therefore, the two exposed portion of an encoder belong to the same simulated Wang tiles. Any two adjacent exposed simulated Wang tiles must be matched in color as we will see in the next paragraph. Thus, the exposed portions of the four encoders simulate the partial tiling of four Wang tiles as illustrated at the bottom right of Figure~\ref{fig_pattern}.
    \item \textbf{The linkers fill the gaps.} After the placement of the locators and encoders, there are still gaps left on the plane. All these gaps are mainly filled by the linkers\footnote{Here, we apply an innovative technique developed by Yoonhu Kim \cite{k25}. In Ollinger's original framework, two sets of significantly different tiles are used to fill the interior or exterior gaps left behind. Kim's technique can decrease the total number of tiles by two.}. The interior gaps inside the locators can always be filled by the linkers. This can be verified by applying Lemma \ref{lem_level3}. The exterior gaps can be filled if and only if the two adjacent simulated edges of the simulated Wang tiles encode the same color. Because by Lemma \ref{lem_level3}, the linkers can either connect a $\{C|A\}$ level-$3$ square with another $\{C|A\}$ level-$3$ square of the encoder, or connect $\{C|B\}$ with $\{C|B\}$. Finally, the tiny gaps between the level-$3$ squares are filled by the tiny filler.
\end{itemize}

We have shown that any plane tiling with the set of $7$ polyominoes has a rigid lattice structure and flexibility for each encoder. As a consequence, any tiling with the set of $7$ polyominoes can be converted to a tiling with the corresponding set of Wang tiles, and any tiling with the corresponding set of Wang tiles can also be converted to a tiling with the set of $7$ polyominoes. By the undecidability of Wang's domino problem (Theorem \ref{thm_berger}), translational tiling of the plane with a set of $7$ orthogonally convex is undecidable.
\end{proof}



\section{Conclusion}\label{sec_conc}

In this paper, we establish the undecidability of translation tiling with a fixed number of orthogonally convex tiles. The natural next step is to find a smaller positive $k\leq 6$ such that the translational tiling problem with a set of $k$ orthogonally convex tiles is undecidable.

A more interesting problem for further study is the decidability or undecidability of translational tiling of a set of convex tiles. As we have mentioned in Section \ref{sec_intro}, tiling with convex polyominoes is trivially decidable; thus we consider tiles of more general shapes.

\begin{Problem}
    Let $k$ be a fixed positive number. Given an arbitrary set of $k$ convex polygons, is there an algorithm to decide whether the plane can be tiled by translated copies of the $k$ convex polygons?
\end{Problem}

For $k=1$, the above problem is decidable because there are complete classifications for the tilability of a single tile under even more general conditions. The translational tiling of the plane with a single tile can always be periodic, even if the tile is nonconvex \cite{bn91,w15}, disconnected \cite{b20}, or a rational polygonal set \cite{dggm04}. As a result, the translational tiling problem with a single tile is decidable. For convex tiles, there is a complete classification of all convex polygons that can tile the plane, even if rotation is allowed (see \cite{zong20} for a concise survey). Is it possible to have an integer $k\geq 2$ such that translational tiling of the plane with a set of $k$ convex polygons is undecidable?


\end{document}